\documentclass{amsart}
\usepackage{amsmath, amsthm,amssymb,amscd,verbatim}
\usepackage[mathscr]{eucal}
\usepackage{enumerate,multicol}
\usepackage{url}
\usepackage{mathtools}
\usepackage{color}

\newtheorem{theorem}{Theorem}[section]
\newtheorem{lemma}[theorem]{Lemma}
\newtheorem{corollary}[theorem]{Corollary}
\newtheorem{proposition}[theorem]{Proposition}

\newtheorem{fact}[theorem]{Fact}

\theoremstyle{definition}
\newtheorem{definition}[theorem]{Definition}
\newtheorem{example}[theorem]{Example}
\newtheorem{remark}[theorem]{Remark}

\def\tp{\operatorname{tp}}

\def\cC{\mathcal C}
\def\cL{\mathcal L}
\def\cU{\mathcal U}
\def\cV{\mathcal V}
\def\C{\mathbb C}
\def\R{\mathbb R}

\def\Z{\mathbb Z}

\def\Th{\textnormal{Th}}
\def\GL{\textnormal{GL}}
\def\U{\textnormal{U}}
\def\T{\textnormal{T}}
\newcommand{\seq}{\subseteq}
\newcommand{\inv}{^{\text{-}1}}
\newcommand{\nv}{\text{-}}

\newcommand{\clqed}{\hfill$\dashv_{\text{\scriptsize{claim}}}$}

\newcommand{\dotminus}{ 
\!\!\buildrel\textstyle~.\over{\hbox{ 
\vrule height3pt depth0pt width0pt}{\smash-} 
}}

\title{Compactifications of pseudofinite and pseudo-amenable groups}

\author{Gabriel Conant}
\address{Department of Mathematics\\
The Ohio State University}
\email{conant.38@osu.edu}

\author{Ehud Hrushovski}
\address{Mathematical Institute\\
University of Oxford}
\email{ehud.hrushovski@maths.ox.ac.uk}

\author{Anand Pillay}
\address{Department of Mathematics\\
University of Notre Dame}
\email{apillay@nd.edu}

\thanks{GC was partially supported by NSF grants DMS-1855503 and DMS-2204787. AP was partially supported by NSF grant DMS-205427.}

\date{June 17, 2025}

\begin{document}

\maketitle

\begin{abstract}
We first give simplified and corrected accounts of some results in \cite{PiRCP} on  compactifications of pseudofinite groups. For instance, we use a classical theorem of Turing \cite{Turing} to give a simplified proof that any definable compactification of a pseudofinite group has an abelian connected component. We then discuss the relationship between Turing's work, the Jordan-Schur Theorem, and a (relatively) more recent result of Kazhdan \cite{Kazh} on approximate homomorphisms, and we use this to widen our scope from finite groups to amenable groups. In particular, we develop a suitable continuous logic framework for dealing with definable homomorphisms from pseudo-amenable groups to compact Lie groups. Together with the stabilizer theorems of \cite{HruAG,MOS}, we obtain a uniform (but non-quantitative) analogue of Bogolyubov's Lemma for sets of positive measure in discrete amenable groups. We conclude with brief remarks on the case of amenable topological groups.
\end{abstract}

\section{Introduction}

This paper starts as a commentary on the third author's paper \cite{PiRCP}, where he studied definable compactifications of pseudofinite groups, partly as a response to a question of Boris Zilber about whether (nonabelian) simple compact Lie groups can be ``structurally approximated" by finite groups. This notion of structural approximation can be translated into relatively standard notions, as was done in \cite{PiRCP}.  In any case, among the main results, Theorem 2.2 of \cite{PiRCP}  was a negative answer to Zilber's question in the case of {\em definable}  or {\em  internal} compactifications. A negative answer for absolute compactifications was later obtained by Nikolov, Schneider, and Thom in \cite{NST}. However, the proofs from \cite{PiRCP} and \cite{NST} both make  use of rather heavy results (described in Section \ref{sec:pillay}).  The starting point of the current paper is to give a short proof in the definable case that uses only classical tools. This is done in Theorem \ref{thm:pillay}. We also use a result of Kazhdan \cite{Kazh} to show that essentially the same proof works for ultraproducts of torsion amenable groups.  

The paper \cite{PiRCP} also contained a result (Theorem 3.1 there)  on the triviality of absolute compactifications of  ultraproducts  of  nonabelian finite simple groups. However, the argument relied on some results from \cite{StoTh} later found to be incorrect. In Section \ref{sec:simple}, we give a correct proof, but making use of the full negative solution to Zilber's question from   \cite{NST}. 

We now turn to the motivation for the rest of the paper.
The third author's result  from \cite{PiRCP} for definable compactifications later played a key role  in the  ``NIP arithmetic regularity lemma" from \cite{CPT}, and subsequent related work from \cite{CoBogo}. These papers focused on proving asymptotic results for subsets of finite groups satisfying various special assumptions, and \cite{PiRCP} was used to access ``Bohr neighborhoods" defined via maps to the torus while working in a  nonabelian environment. In order to make various technical aspects of these arguments feasible for transfer with {\L}o{\'s}'s Theorem, \cite{CPT} and \cite{CoBogo} also used a result of Alekseev, Glebski\u{\i}, and Gordon \cite{AGG} on approximate homomorphisms from finite groups to compact Lie groups. In fact, this result was first proved by Kazhdan \cite{Kazh} in the setting of approximate homomorphisms from amenable groups to unitary groups. 
 
 In the context of (arithmetic) regularity,  a subset (of a group) can be viewed as a $\{0,1\}$-valued function, but it is often convenient or necessary to expand one's viewpoint to include $[0,1]$-valued functions. However, the functional setting only compounds the technical challenges alluded to above. Therefore, here we develop a more robust foundation for applying Kazhdan's approximation theorem in these settings, which circumvents the somewhat ad hoc approach in previous sources. In particular, we will view definable compactifications as objects of continuous logic, and prove the following result. 

\begin{theorem}\label{thm:main1}
Let $\cL$ be a language in either classical or continuous logic, which extends the language of groups.
Let $G=\prod_{\cU}G_i$ be an ultraproduct of a family $\{G_i:i<\omega\}$ of $\cL$-structures expanding amenable groups. Suppose $\tau\colon G\to K$ is a definable homomorphism, with $K$ a compact Lie group. Then there are homomorphisms $\tau_i\colon G_i\to K$, for $i<\omega$, such that $\tau=\lim_{\cU}\tau_i$.
\end{theorem}

See Theorem \ref{thm:mainLKt} for a restatement and proof of this result. Our approach first passes through the general situation of an ultraproduct $M=\prod_{\cU}M_i$ (with $\cL$ arbitrary) and a definable map $f\colon M\to C$ for some (compact) space $C$. In Sections \ref{sec:maps1} and \ref{sec:maps2} we take the opportunity to spell out some possibly folkloric facts about this situation, with an emphasis on the difference between when $\cL$ is classical versus continuous.

The key advantage to Theorem \ref{thm:main1} is the potential to treat $\tau$ syntactically as a function symbol in continuous logic, and thus transfer first-order statements to the $\tau_i$'s  directly without further approximation lemmas (c.f., \cite{CPT,CoBogo}). More precisely, under some  extra assumptions, one can  realize the two-sorted structure $(G,K,\tau)$ as an ultraproduct $\prod_{\cU}(G_i,K,\tau_i)$ (see Corollary \ref{cor:mainLKt}). 

Our primary motivation for allowing a continuous language $\cL$  in Theorem \ref{thm:main1} is for use in later applications (by the first and third author) to arithmetic regularity theorems in the functional setting discussed above; see \cite{CP-AVSAR}. However, as an illustration in the case when $\cL$ is a language in classical first-order logic, we will combine Theorem \ref{thm:main1} with the stabilizer theorem of the second author \cite{HruAG} to prove the following version of Bogolyubov's Lemma for positive measure sets in amenable groups (see Theorem \ref{thm:amenBogo}).

\begin{theorem}\label{thm:amenBogo0}
Fix $\alpha>0$ and a function $\epsilon\colon\R^+\times\Z^+\to \R^+$. 
Let $G$ be an amenable group with a left-invariant measure $\mu$. Suppose $A\seq G$ is such that $\mu(A)\geq \alpha$. Then there is a $(\delta,\U(n))$-Bohr neighborhood $B\seq G$, with $\delta\inv,n\leq O_{\alpha,\epsilon}(1)$, such that:
\begin{enumerate}[$(i)$]
\item  $B\seq (AA\inv)^2$, 
\item $AAA\inv$ contains a translate of $B$, and 
\item $\mu(B\backslash AA\inv)<\epsilon(\delta,n)\mu(B)$. 
\end{enumerate}
\end{theorem}

Bohr neighborhoods are defined and discussed in greater generality at the start of Section \ref{sec:unitary}. 
Here $\U(n)$ denotes the unitary group of degree $n$. If one restricts to the special cases of  finite groups or abelian groups, then $\U(n)$ can be replaced by the $n$-dimensional torus $\T(n)$ (possibly at the cost of passing to a normal subgroup of uniformly bounded index; see Proposition \ref{prop:UtoT}).  When $G$ is finite and abelian,  conclusion $(i)$ of the previous theorem is the statement of Bogolyubov's Lemma, as originally proved by Ruzsa in \cite{Ruz94} (with explicit  bounds). The extension to arbitrary finite groups was proved by the first author in \cite{CoBogo} (without explicit bounds).  In \cite{BjGr}, Bjorklund and Griesmer obtain quantitative  results closely related to conclusion $(ii)$ of Theorem \ref{thm:amenBogo0} with  bounds on the number of irreducible representations defining the relevant Bohr neighborhood, but not on their degrees.

 We will conclude Section \ref{sec:unitary} with some preliminary remarks on the extension of these results to  amenable topological amenable groups (see Proposition \ref{prop:topamen}).

\subsection*{Acknowledgements} We thank the referee for providing several  corrections and suggestions for revision. The first author thanks James Hanson for many helpful discussions, and for pointing out the need to work with absolute neighborhood retracts in Proposition \ref{prop:CLmaps}. 

\section{Definable compactifications, Turing, and Kazhdan}\label{sec:pillay}

In \cite{PiRCP}, the third author proved that if $G$ is a pseudofinite first-order structure expanding a group, then the connected component of any definable (group) compactification of $G$ is abelian. Shortly after, this result was generalized to \emph{any} compactification of an (abstract) pseudofinite group by Nikolov, Schneider, and Thom \cite{NST}. The argument in \cite{PiRCP} is short, but relies on the main structure theorem of Breuillard, Green, and Tao \cite{BGT} for finite approximate subgroups of arbitrary groups. The proof of the more general result in \cite{NST} is also not very long, but relies on an intricate result of Nikolov and Segal  \cite[Theorem 1.2]{NiSe} on  commutators in finite groups, which in turn relies on the classification of finite simple groups.

In this section, we give a short proof of the above result in the definable case, which uses much more classical tools (compared to \cite{BGT} and \cite{NiSe}). In particular, the proof uses only the Peter-Weyl Theorem and a result of Alan Turing on finite approximations of Lie groups, which itself uses Jordan's Theorem (any finite subgroup of a general linear group contains a large abelian subgroup). This strategy lines up nicely with the proof in \cite{PiRCP} from Breuillard-Green-Tao \cite{BGT}. Indeed, the ``nonstandard" aspects of \cite{BGT} involve work of Gleason and Yamabe on locally compact groups, as well as commutator estimate techniques inspired by the proof of Jordan's Theorem due to Bieberbach and Frobenius. As a conclusion, they show that Hrushovski's  ``Lie model" \cite{HruAG} has a nilpotent connected component. From a distance, this explains the emergence of nilpotent structure (``nilprogressions") in finite approximate subgroups of arbitrary groups. However, in \cite{PiRCP}, these results are applied in the special case of \emph{compact} groups, in which case the connected component of the Lie model must be abelian. Thus it makes sense that there would be a more direct proof using Peter-Weyl and the work of Turing which, roughly speaking, are ``compact" analogues of Gleason-Yamabe and the Bieberbach-Frobenius steps of \cite{BGT}, respectively. 

To state Turing's result, we need the following definition.

\begin{definition}
Let $G$ be a group and suppose $K$ is a metrizable group with a bi-invariant metric $d$. Then a map $f\colon G\to K$ is an \textbf{$\epsilon$-approximate homomorphism} if $d(f(xy),f(x)f(y))\leq \epsilon$ for all $x,y\in G$. The (metric) group $K$ is \textbf{finitely $\epsilon$-approximable} if there is an $\epsilon$-approximate homomorphism from a finite group to $K$, whose image is an $\epsilon$-net in $K$. We say $K$ is \textbf{finitely approximable} if it is finitely $\epsilon$-approximable for all $\epsilon>0$. 
\end{definition}

In the context of the previous definition, one can check that finite approximability of $K$ does not depend on the choice of invariant metric $d$.

Recall that any compact Lie group can be equipped with some (not necessarily unique) bi-invariant metric compatible with the topology. 

\begin{theorem}[Turing \cite{Turing}]\label{thm:turing}
The connected component of any finitely approximable compact  Lie group is abelian.
\end{theorem}

\begin{remark}\label{rem:turing}
The previous result does not appear in \cite{Turing} exactly as we have stated it. So we take a moment to clarify two key differences.

First, Turing works exclusively with connected Lie groups, and thus only explicitly proves Theorem \ref{thm:turing} in the connected case. However, it is not hard to adjust his proof to obtain the more general statement. In particular, given a finitely approximable compact Lie group $K$, Turing's proof produces a set $X\seq K$ of positive (Haar) measure such that for any $x\in X$, the centralizer $C(x)$ has positive  measure. Up until this point, the proof does not use connectedness, and assuming $K$ is connected he then concludes $K$ is abelian. Without assuming connectedness, we can argue similarly to prove $K^0$ is abelian. First, for any $x\in X$, $C(x)$ is a closed subgroup of finite index (since it has positive measure), and thus contains $K^0$. So $X$ is contained in the closed subgroup $C(K^0)\coloneqq\bigcap_{a\in K^0}C(a)$. Since $X$ has positive measure, we again conclude that $C(K^0)$ has finite index, and hence contains $K^0$. Thus $K^0$ is abelian.

Second, Turing's definition of finite $\epsilon$-approximability requires injectivity of the $\epsilon$-approximate homomorphism involved. But this does not affect finite approximability for a compact Lie group $K$. In particular, we may clearly assume $K$ is infinite and thus positive dimensional. Now suppose $f\colon G\to K$ is an $\epsilon$-approximate homomorphism with $G$ finite, and $f(G)$ is an $\epsilon$-net. Since $K$ has positive dimension,  any open ball  is infinite. So for any $\epsilon^*>\epsilon$,  by perturbing $f$ one can construct an injective $\epsilon^*$-approximate homomorphism $f^*\colon G\to K$ with $f^*(G)\supseteq f(G)$. 
\end{remark}

Toward the end of this section, we will sketch a short proof of Theorem \ref{thm:turing} using Peter-Weyl, Jordan, and a result of Kazhdan \cite{Kazh} on approximate  homomorphisms. But first let us give the  proof of the aforementioned theorem from \cite{PiRCP}. 

 Let $M$ be a first-order structure. Given a topological space $X$, we say that a map $f\colon M^x\to X$ is \emph{definable} if for any closed $C\seq X$ and open $U\seq X$, with $C\seq U$, there is a definable set $D\seq M^x$ such that $f\inv(C)\seq D\seq f\inv(U)$. Equivalently,  $f$ is definable if induces a (necessarily unique) continuous function $f^*\colon S_x(M)\to X$ such that $f^*(\tp(a/M))=f(a)$ for all $a\in M^x$.\footnote{Here $S_x(M)$ denotes the space of complete types over $M$ in free variables $x$, i.e., the Stone space of the Boolean algebra of definable subsets of $M^x$.}

 Now suppose $G$ is a first-order structure expanding a group. A \emph{(group) compactification} of $G$ is a group homomorphism $\tau\colon G\to K$ with dense image, where $K$ is a compact Hausdorff group. We will focus on definable compactifications. The canonical example (which is the focus of \cite{PiRCP}) is when $G$ is sufficiently saturated, $K=G/G^{00}_M$ (with the logic topology) for some small model $M\prec G$, and $\tau$ is the quotient map.

\begin{theorem}[Pillay \cite{PiRCP}]\label{thm:pillay}
Suppose $G$ is a pseudofinite first-order expansion of a group and $\tau\colon G\to K$ is a definable compactification. Then the connected component of $K$ is abelian.
\end{theorem}
\begin{proof}
By the Peter-Weyl Theorem, $K$ is an inverse limit of compact Lie groups, with surjective projection maps. Since the composition of $\tau$ with any (continuous) projection map is still a definable compactification, it altogether suffices to assume that $K$ is a compact Lie group. Equip $K$ with some choice of a bi-invariant compatible metric $d$. We will show that $K$ is finitely approximable,\footnote{This argument follows the same idea as that of  \cite[Lemma 5.4]{CPT}, which itself was inspired by material from unpublished notes of the second author.} and thus has an abelian connected component by Theorem \ref{thm:turing}.

Fix $\epsilon>0$ and let $\delta=\epsilon/6$. Let $\Lambda$ be a finite $\delta$-net in $K$. Given $t\in\Lambda$, let $C_t$ (resp., $U_t$) be the closed (resp., open) ball  in $K$ of radius $\delta$ (resp., $2\delta$) centered at $t$. So we have $C_t\seq U_t$ for any $t\in \Lambda$. Since $\tau$ is definable, there is a definable set $X_t\seq G$ such that $\tau\inv(C_t)\seq X_t\seq \tau\inv(U_t)$. Since $\Lambda$ is a $\delta$-net, we have $K=\bigcup_{t\in\Lambda}C_t$, and thus $G=\bigcup_{t\in\Lambda}X_t$. Note also that each $X_t$ is nonempty since $\tau(G)$ is dense in $K$.   Moreover, if $a\in X_t$ then $d(\tau(a),t)<2\delta$ since $X_t\seq\tau\inv(U_t)$. From this, we obtain the following properties (explained below):
\begin{enumerate}[$(i)$]
\item For any $s,t\in\Lambda$, if $X_s\cap X_t\neq\emptyset$ then $d(s,t)<4\delta$.
\item For any  $r,s,t\in\Lambda$, if there are $a,b\in G$ such that $a\in X_r$, $b\in X_s$, and $ab\in X_t$, then $d(rs,t)<6\delta=\epsilon$.
\end{enumerate}

Indeed, for $(i)$ if there is some $a\in X_s\cap X_t$ then $d(s,t)\leq d(s,\tau(a))+d(\tau(a),t)<4\delta$. Similarly, if $r,s,t\in\Lambda$ and $a,b\in G$ are as in $(ii)$, then
\begin{align*}
d(rs,t) &\leq d(rs,r\tau(b))+d(r\tau(b),\tau(a)\tau(b))+d(\tau(a)\tau(b),t)\\
&= d(s,\tau(b))+d(r,\tau(a))+d(\tau(ab),t)<6\delta.
\end{align*}

Since $\Lambda$ is finite, we can express the previous data as a first-order sentence (with existential quantifiers over parameters needed to define $X_t$). Thus, since $G$ is pseudofinite, we obtain a finite group $G_\epsilon$ and a cover $G_\epsilon=\bigcup_{t\in\Lambda}X_{t,\epsilon}$ by nonempty subsets satisfying the analogues of $(i)$ and $(ii)$. Let $f\colon G_\epsilon \to \Lambda$ be any function such that $a\in X_{f(a),\epsilon}$ for all $a\in G_\epsilon$.  It follows from $(ii)$ that $f$ is an $\epsilon$-approximate homomorphism. Moreover, given  $t\in\Lambda$, we know there is some $a\in X_{t,\epsilon}$, and so $d(f(a),t)<4\delta$ by $(i)$. Since $\Lambda$ is a $\delta$-net in $K$, it follows that $f(G_\epsilon)$ is a $5\delta$-net (and thus an $\epsilon$-net) in $K$. 
So $f$ witnesses that $K$ is finitely $\epsilon$-approximable.
 \end{proof}

Next we aim to give a proof of Turing's theorem using  a result of Kazhdan on approximate homomorphisms, which was later re-proved by Alekseev, Glebski\u{\i}, and Gordon in a  stronger form (see Remark \ref{rem:Kaz}). This result will also allow us to widen our perspective  to amenable groups. Recall that a (discrete) group $G$ is \emph{amenable} if it admits a left-invariant finitely additive probability measure on all subsets. 
 
 \begin{theorem}[Kazhdan \cite{Kazh}; Alekseev, Glebski\u{\i}, Gordon \cite{AGG}]\label{thm:Kaz}
 Let $K$ be a compact Lie group. Then there is a bi-invariant compatible metric $d$ on $K$ and some $\epsilon_K>0$ such that for any amenable group $G$ and any $\epsilon$-approximate homomorphism $f\colon G\to K$, with $\epsilon<\epsilon_K$, there is a homomorphism $\tau\colon G\to K$ satisfying $d(f(a),\tau(a))\leq 2\epsilon$ for all $a\in G$.
\end{theorem}

\begin{remark}\label{rem:Kaz}
The previous theorem is proved by Kazhdan in \cite{Kazh} for $K$ a unitary group with metric induced from the operator norm, in which case one can take $\epsilon_K=1/200$.\footnote{To prevent possible confusion, we note that an $\epsilon$-approximate homomorphism here corresponds to a ``$2\epsilon$-homomorphism" in \cite{Kazh}.} Alekseev, Glebski\u{\i}, and Gordon \cite{AGG} modify this result for an arbitrary compact Lie group $K$. The proof in \cite{AGG} is written for finite $G$, and the authors remark that the argument goes through for amenable $G$, as in Kazhdan's paper. We also note that Kazhdan's proof applies when $G$ is an amenable \emph{topological} group, in which case one assumes the approximate homomorphism is continuous (and obtains continuity in the resulting homomorphism).  
\end{remark}

One can  view Theorem \ref{thm:Kaz}   as a ``rigidity theorem" in the sense of additive combinatorics and combinatorial number theory (see also Tao's ICM lecture on rigidity phenomena in mathematics \cite{TaoICM}). The general idea is that  mathematical objects exhibiting approximate structure can be  approximated by objects with perfect structure. As a specific instance of this, Theorem \ref{thm:Kaz} says that an approximate homomorphism (of an amenable group to a compact Lie group) can be approximated by a genuine homomorphism. 

We now use Theorem \ref{thm:Kaz} to give a short proof of Turing's theorem, which still emphasizes the fundamental role of Jordan's Theorem.

\begin{proof}[\textnormal{\textbf{Proof of Theorem \ref{thm:turing}}}]
Suppose $K$ is a finitely approximable compact Lie group. Then by Theorem \ref{thm:Kaz}  we have, for all $m>0$, a finite subgroup $G_m\leq K$ which is a $\frac{1}{m}$-net in $K$. By the Peter-Weyl Theorem, $K$ is isomorphic to a closed subgroup of $\GL_n(\C)$ for some $n\geq 1$ (see \cite[Corollary IV.4.22]{Knappbook}).  So we may view each $G_m$ as a finite subgroup of $\GL_n(\C)$.  By Jordan's Theorem \cite{Jordan}, there is an integer $d\geq 1$ (depending only on $n$) such that for all $m>0$, $G_m$ contains an abelian subgroup $H_m$ of index at most $d$ (in $G_m$). Pick a nonprincipal ultrafilter $\cU$ on $\Z^+$, and let $G=\prod_{\cU}G_m$ and $H=\prod_{\cU} H_m$. Then $H$ is an abelian subgroup of $G$ of index at most $d$. Moreover, given $g\in G$, if we pick a representative $g=[(g_m)]_{\cU}$ then $(g_m)$ is a sequence in $K$ with a well-defined ultralimit  $\tau(g)\coloneqq\lim_{\cU}g_m$ in $K$ depending only on $g$. A routine verification then shows that $\tau\colon G\to K$ is a group homomorphism. Moreover, since each $G_m$ is a $\frac{1}{m}$-net and $\cU$ is nonprincipal, it follows that $\tau$ is surjective. Therefore $\tau(H)$ is an abelian subgroup of $K$ of index at most $d$. It follows that $K^0$ is contained in the closure of $\tau(H)$, which is still abelian.
\end{proof}

\begin{remark}
In the above proof of Theorem \ref{thm:turing}, we needed to replace approximate homomorphisms by actual homomorphisms in order to obtain genuine finite subgroups of $\GL_n(\C)$. In doing so, we realized $K$ as a compactification of $\prod_{\cU}G_m$. So it is worth pointing out that this latter conclusion can be obtained from approximate homomorphisms alone. In particular, suppose $K$ is a finitely approximable compact metric group. Then for all $m>0$ we have a finite group $G_m$ and a $\frac{1}{m}$-approximate homomorphism $f_m\colon G_m\to K$ whose image is a $\frac{1}{m}$-net. As in the above proof, we can choose a nonprincipal ultrafilter $\cU$ and define an ultralimit map $\tau\colon \prod_{\cU}G_m\to K$ such that $\tau(g)=\lim_{\cU} f_m(g_m)$. Then a similar routine verification shows that $\tau$ is a surjective homomorphism. In fact, $\tau$ is even definable in a suitable expansion of $\prod_{\cU}G_m$ by internal sets (as in Example \ref{ex:cyclic} below). Altogether, $K$ is a definable compactification of a pseudofinite group. So the fact that $K^0$ is abelian follows from   Theorem \ref{thm:pillay} (proved originally in   \cite{PiRCP} using different tools), and also the generalization to arbitrary compactifications from \cite{NST}. Indeed, the authors of \cite{NST} explicitly note that their work yields an alternate proof of Turing's theorem.
\end{remark}

Via the previous proof of Theorem \ref{thm:turing}, we obtain a proof of Theorem \ref{thm:pillay} that uses only the Peter-Weyl Theorem, Jordan's Theorem, and Theorem \ref{thm:Kaz}. Of these tools, only Jordan's Theorem directly references finite groups.  Schur \cite{SchurJST} later proved that any finitely-generated torsion subgroup of $\GL_n(\C)$ is finite. Using this, one can deduce the following generalization of Jordan's theorem to torsion groups (see \cite[Theorem 36.14]{CurRein} or \cite[Chapter 11]{TaoH5P}).

\begin{fact}[Jordan-Schur Theorem]\label{fact:JST}
For any $n\geq 1$ there is some $d\geq 1$ such that any torsion subgroup of $\GL_n(\C)$ contains an abelian subgroup of index at most $d$.
\end{fact}

The same proof then yields the following generalization of Theorem \ref{thm:pillay}.

\begin{theorem}\label{thm:Pilgen}
Suppose $G$ is a first-order structure elementarily equivalent to an ultraproduct of amenable torsion  groups, and $\tau\colon G\to K$ is a definable compactification. Then the connected component of $K$ is abelian.
\end{theorem}

As a segue into the next section, we conclude this section with some consequences for ``pseudo-simple" groups. First we recall a well known fact.

\begin{proposition}\label{prop:easysimple}
Suppose $G$ is an infinite first-order structure elementarily equivalent to an ultraproduct of simple groups and let $\tau\colon G\to K$ be a definable compactification of $G$. Then $K$ is connected.
\end{proposition}
\begin{proof}
If $K$ is not connected, then it has a proper clopen normal subgroup $C$ of finite index. So $\tau\inv(C)$ is a proper (by density of $\tau(G)$ in $K$) definable normal subgroup of $G$ of finite index. But this contradicts the assumptions on $G$.
 \end{proof}

\begin{corollary}\label{cor:fsg-def}
Suppose $G$ is an infinite first-order structure elementarily equivalent to an  ultraproduct of  simple amenable torsion groups, and let $\tau\colon G\to K$ be a definable compactification of $G$. Then:
\begin{enumerate}[$(a)$]
\item $K$ is abelian and connected.
\item If $G$ is nonabelian then $K$ is trivial.
\end{enumerate}
\end{corollary}
\begin{proof}
Part $(a)$ is immediate from Proposition \ref{prop:easysimple} and Theorem \ref{thm:Pilgen}. For part $(b)$, assume $G$ is nonabelian. As in the proof of Theorem \ref{thm:pillay}, we may use the Peter-Weyl Theorem to reduce to the case that $K$ is a compact Lie group. Let $\epsilon>0$ be arbitrary. Following the proof of Theorem \ref{thm:pillay}, we  find a nonabelian simple amenable torsion group $H$ and an $\epsilon$-approximate homomorphism $f\colon H\to K$ whose image is an $\epsilon$-net in $K$.  By Theorem \ref{thm:Kaz}, we can replace $f$ with a genuine homomorphism $\tau'\colon H\to K$ whose image is a $3\epsilon$-net in $K$. Since $H$ is nonabelian and simple, and $K$ is abelian, it follows that $\tau'$ is trivial. So the identity is a $3\epsilon$-net in $K$. Since $\epsilon$ was arbitrary, we conclude that $K$ is trivial. 
\end{proof}

As shown previously by Palac\'{i}n \cite[Theorem 4.8]{PalPFS}, if one restricts to  finite groups (rather than amenable torsion groups) then  Corollary \ref{cor:fsg-def}$(b)$ can be deduced from results of Gowers \cite{GowQRG} on quasirandom groups (which will be discussed in Section \ref{sec:unitary}).  In fact, it turns out that in the pseudofinite case the previous corollary also holds for \emph{absolute} compactifications. However in this case the proofs require more high-powered tools such as the classification of finite simple groups. We will elaborate on this in Section \ref{sec:simple}.  
 
 \begin{example}\label{ex:cyclic}
 For completeness, we describe a concrete example showing that the nonabelian assumption in part $(b)$ of Corollary \ref{cor:fsg-def} is necessary. Given a prime number $p$, define $\tau_p\colon \Z/p\Z\to S^1$ so that $\tau_p(x)=e^{2\pi i x/p}$. Let $\cU$ be a nonprincipal ultrafilter on the set of prime numbers and let $G=\prod_{\cU}\Z/p\Z$ (as a classical structure in the language of groups). Let $\tau=\lim_{\cU}\tau_p\colon G\to S^1$. Then $\tau$ is a surjective homomorphism. Given a prime $p$, an $\epsilon>0$, and some $z\in S^1$, set $A_{p,\epsilon}(z)=\{g\in \Z/p\Z:d(\tau_p(g),z)<\epsilon\}$. Then $\tau$ is definable in the expansion of $G$ by predicates for the  internal sets $A_\epsilon(z)=\prod_{\cU}A_{p,\epsilon}(z)$ for all $\epsilon>0$ and $z\in S^1$.
 \end{example}

\section{Ultraproducts of finite simple groups}\label{sec:simple}

The goal of this section is to prove a result about absolute compactifications of  pseudofinite groups elementarily equivalent to an ultraproduct of finite simple groups.  Our primary motivation is to supply a correct proof of Theorem \ref{thm:fsg-gen}$(b)$ below, which appeared previously in work of the third author \cite[Theorem 3.1]{PiRCP}. However the proof there used a result of Stolz and Thom \cite{StoTh} on the lattice of normal subgroups of an ultraproduct of finite simple groups, which was later found to be incorrect (see \cite{SchTh-lattice}). 

In this section, we  follow \cite{MarZelFSG} and use the notation $G^n$ for the subgroup of a group $G$ generated by all powers $a^n$ for $a\in G$. 
We will make use of several known facts about groups, which we now state. The first is really a fusion of two results, so we provide a short explanation. 

\begin{theorem}[Martinez, Zelmanov  \cite{MarZelFSG}; Saxl, Wilson \cite{SaxWilFSG}; Babai, Goodman,  Pyber \cite{BaGoPy}]\label{thm:fsg1}
For any $n\geq 1$ there is some $k\geq 1$ such that if $G$ is a finite simple group of size at least $k$, then every element of $G$ is a product of $k$ elements in $G^n$ (i.e.,  $G=G^n\,\cdot\stackrel{k}{\ldots}\cdot\, G^n$).
\end{theorem}
\begin{proof}[Explanation]
It was shown in \cite{MarZelFSG} and \cite{SaxWilFSG} that for any $n\geq 1$ there is some $k\geq 1$ such that if $G$ is a finite simple group then either  $G=G^n\,\cdot\stackrel{k}{\ldots}\cdot\, G^n$ or $G$ has exponent $n$. It was shown in \cite{BaGoPy} that for any $n\geq 1$ there are only finitely many finite simple groups of exponent $n$. 
\end{proof}

Next, we quote  \cite[Proposition 2.4]{WilSPfG}.

\begin{theorem}[Wilson \cite{WilSPfG}]\label{thm:fsg2}
There is an integer $m\geq 1$ such that if $G$ is a nonabelian finite simple group then any element of $G$ is a product of at most $m$ commutators. 
\end{theorem}

The previous result is based on \emph{Ore's Conjecture}, which asserts that one can take $m=1$. This was later proved by Liebeck, O'Brien, Shalev, and Tiep \cite{LOST}. 

Finally, we state the generalization of Theorem \ref{thm:pillay} for absolute compactifications.

\begin{theorem}[Nikolov, Schneider, Thom \cite{NST}]\label{thm:NST}
Suppose $G$ is elementarily equivalent (in the group language) to an ultraproduct of finite groups, and $\tau\colon G\to K$ is a compactification of $G$. Then the connected component of $K$ is abelian.
\end{theorem}

It is worth pointing out that all three of the above theorems rely on the classification of finite simple groups (for Theorem \ref{thm:NST} the dependence goes through a result of Nikolov and Segal \cite[Theorem 1.2]{NiSe}).  Using these tools, we can prove a direct analogue of Corollary \ref{cor:fsg-def} for absolute compactifications of ultraproducts of finite simple groups.

\begin{theorem}\label{thm:fsg-gen}
Suppose $G$ is an infinite group elementarily equivalent (in the group language) to an  ultraproduct of finite simple groups, and $\tau\colon G\to K$ is a compactification of $G$. Then:
\begin{enumerate}[$(a)$]
\item $K$ is connected, and thus abelian by Theorem \ref{thm:NST}.
\item If $G$ is nonabelian then $K$ is trivial.
\end{enumerate}
\end{theorem}
\begin{proof}
Part $(a)$. If $K$ is not connected, then we can find a proper finite-index clopen subgroup $C\leq K$, which yields a proper  finite-index subgroup $\tau\inv(C)$ of $G$. So it suffices to show that $G$ has no proper finite-index subgroups. 

Let $H\leq G$ be a finite-index subgroup of $G$. We want to show that $G=H$. Without loss of generality, we may assume $H$ is normal. Let $n$ be the index of $H$ in $G$, and let $k$ be as in Theorem \ref{thm:fsg1}. Since $G$ is infinite, it is elementarily equivalent to an ultraproduct of finite simple groups of size at least $k$. So Theorem \ref{thm:fsg1} implies $G^n\,\cdot\stackrel{k}{\ldots}\cdot\, G^n=G$. But $G^n\seq H$ since $G/H$ is a finite group of size $n$. So $G=H$.  

Part $(b)$. Since $K$ is abelian, $\ker\tau$ contains the derived subgroup of $G'$ of $G$, which is all of $G$ by Theorem \ref{thm:fsg2}.  
\end{proof}

As previously discussed, part $(b)$ provides a correct proof of a claim originally from \cite{PiRCP}. So it is slightly amusing that the proof relies on Theorem \ref{thm:NST}, which was conjectured in \cite{PiRCP}.

\section{Maps to compact sorts}

We start this section with some motivation. Let $G$ be a pseudofinite first-order structure expanding a group, and suppose $\tau\colon G\to K$ is a definable compactification, with $K$ a compact Lie group. Then we know from Theorem \ref{thm:pillay} that the connected component of $K$ is a compact connected \emph{abelian} Lie group, and thus isomorphic to a finite-dimensional torus $\T(n)=(\R/\Z)^n$. Therefore if $U\seq \T(n)$ is some open neighborhood of the identity, then the preimage $\tau\inv(U)$ is essentially a Bohr neighborhood in $G$ in the  sense of additive combinatorics \cite{TaoVu}.  Now, although $\tau\inv(U)$ can be approximated by definable sets, it is not itself definable. So there is some work required to transfer statements about $\tau\inv(U)$ to statements about Bohr neighborhoods in finite groups. In \cite{CPT}, this was done by first approximating $\tau\inv(U)$ by definable ``approximate Bohr neighborhoods", and then using Theorem \ref{thm:Kaz} to recover genuine Bohr neighborhoods. While effective, this process was somewhat cumbersome and inflexible. 

In this section we take a different approach, which uses Theorem \ref{thm:Kaz} at the onset in order to realize the triple $(G,K,\tau)$ as an ultraproduct of finite groups equipped with a homomorphism to $K$. This will ultimately be formalized  in the setting of continuous logic. So for clarity, we will distinguish between ``classical" and ``continuous" first-order logic throughout this section.

\begin{definition}
Let $\cL$ be a first-order language in either classical or continuous logic, and let $T$ be an arbitrary $\cL$-theory.  Given a a tuple $x$ of variables, we let $S^{\cL}_x(T)$ denote the space of complete $\cL$-types in variables $x$ consistent with $T$ (i.e., the Stone space of the Boolean algebra of $\cL$-formulas in free variables $x$ modulo equivalence in $T$). 
\end{definition}

We emphasize that in the previous definition, $T$ is \emph{not} assumed to be complete. So,
 for example, $S^{\cL}_x(\emptyset)$ is the space of all complete $\cL$-types realized in some $\cL$-structure.

\subsection{Definable maps on ultraproducts in classical logic}\label{sec:maps1}

Given a topological space $X$, recall that a subspace $A\seq X$ is a \emph{retract of $X$} if there is a retraction from $X$ to $A$, i.e., a continuous function from $X$ to $A$ whose restriction to $A$ is the identity. Note that if $A$ is a retract of $X$, then any continuous function from $A$ to some other space $C$ can be extended to a continuous function from $X$ to $C$ by composing with a retraction from $X$ to $A$. We will use the following topological result (see \cite[\S 26.II Corollary 2]{Kuratowski}). 

\begin{fact}\label{fact:Kura}
Suppose $X$ is a separable metrizable space with a basis of clopen sets. Then any closed subset of $X$ is  a retract of $X$. 
\end{fact}

 We apply the previous fact as follows.

\begin{lemma}\label{lem:Kura}
Let $\cL$ be a countable language in classical logic, and suppose $T_0\seq T$ are $\cL$-theories. Then $S^{\cL}_x(T)$ is a retract of $S^{\cL}_x(T_0)$. Thus any continuous function from $S^{\cL}_x(T)$ to a topological space $C$ can be extended to a continuous function from $S^{\cL}_x(T_0)$ to $C$. 
\end{lemma}
\begin{proof}
 Since $\cL$ is countable, $S^{\cL}_x(T_0)$ satisfies the assumptions of Fact \ref{fact:Kura}. Clearly $S^{\cL}_x(T)$ is a closed subset of $S^{\cL}_x(T_0)$.
\end{proof}

\begin{proposition}\label{prop:DLmaps}
Let $\cL$ be a countable language in classical logic, and let $M=\prod_{\cU}M_i$ be an ultraproduct of a family $\{M_i:i\in I\}$ of $\cL$-structures. Suppose $C$ is a compact Hausdorff space and $f\colon M^x\to C$ is  $\emptyset$-definable. Then for each $i\in I$, there is an $\emptyset$-definable function $f_i\colon M^x_i\to C$ such that $f=\lim_{\cU}f_i$.
\end{proposition}
\begin{proof}
We can view $f$ as a continuous function from $S^{\cL}_x(T)$ to $C$ where $T=\Th(M)$. By Lemma \ref{lem:Kura} (with $T_0=\emptyset$), $f$ extends to a continuous function $f'\colon S^{\cL}_x(\emptyset)\to C$. For each $i\in I$, define $f_i\colon M^x_i\to C$ so that for $a\in M^x$, $f_i(a)=f'(\tp_{\cL}(a))$. It is then straightforward to check that $f=\lim_{\cU}f_i$. 
\end{proof}

\begin{remark}\label{rem:DLmaps}$~$
\begin{enumerate}[$(1)$]
\item It is an elementary exercise to show that if $X$ is a Hausdorff space for which every closed subset is a retract, then $X$ has a basis of clopen sets (i.e., $X$ is \emph{zero-dimensional}). Thus Fact \ref{fact:Kura} provides a characterization of zero-dimensionality for separable metrizable spaces in terms of retracts.
\item A Hausdorff space is called \textit{ultraparacompact} if any open cover can be refined by a partition into clopen sets (examples include Stone spaces and zero-dimensional separable metrizable spaces). Given an ultraparacompact space $X$, Fact \ref{fact:Kura} can be adapted to say that any closed completely metrizable subset of  $X$ is a retract of $X$. This follows from a result of Ellis \cite{EllisRL}, which says that if $X$ is ultraparacompact and $A\seq X$ is closed, then any continuous function from $A$ to a completely metrizable space $Y$ extends to a continuous function from $X$ to $Y$. Note that this result directly yields Lemma \ref{lem:Kura} when $C$ is completely metrizable, and without assuming $\cL$ is countable. Consequently, we obtain a variation of Proposition \ref{prop:DLmaps} where $\cL$ is of arbitrary cardinality and $C$ is compact, Hausdorff, and (completely) metrizable. However, in this case, $C$ is second-countable and thus any  definable function from an $\cL$-structure to $C$ is $\emptyset$-definable with respect to a countable sublanguage of $\cL$ expanded by countably many constants. Therefore this variation of Proposition \ref{prop:DLmaps} already follows from the proof for countable languages.
\end{enumerate}
\end{remark}

\subsection{Definable maps on ultraproducts in continuous logic}\label{sec:maps2}

Next we wish to obtain a version of Proposition \ref{prop:DLmaps} for metric structures. We assume familiarity with the basics of continuous logic. See \cite{BBHU} for an introduction.
Since type spaces in continuous logic are no longer totally disconnected, we will need to take a different approach than the classical case (in light of Remark \ref{rem:DLmaps}$(1)$). 

Recall that a metrizable topological space $C$ is called an \emph{absolute neighborhood retract} if for any embedding of $C$ as a closed subspace of some metrizable space $Y$, there is an open set $U\seq Y$ containing $C$ such that $C$ is a retract of $U$. 

\begin{proposition}\label{prop:CLmaps}
Let $\cL$ be a language in continuous logic, and let $M=\prod_{\cU}M_i$ be an ultraproduct of a family $\{M_i:i\in I\}$ of $\cL$-structures. Suppose $C$ is a compact absolute neighborhood retract and $f\colon M^x\to C$ is $\emptyset$-definable. Then for each $i\in I$, there is an $\emptyset$-definable function $f_i\colon M^x_i\to C$ such that $f=\lim_{\cU}f_i$.
\end{proposition}
\begin{proof}
Since $C$ is compact and metrizable, we can embed it as a closed subset of $[0,1]^\omega$. %
View $f$ as a continuous function from $S^{\cL}_x(T)$ to $[0,1]^\omega$, where $T=\Th(M)$. By the Tietze Extension Theorem, we can extend $f$ to a continuous function $f'\colon S^{\cL}_x(\emptyset)\to [0,1]^\omega$. For $i\in I$, define $f'_i\colon M_i^x\to [0,1]^\omega$ so that for $a\in M_i^x$, $f'_i(a)=f'(\tp(a))$. Then $f=\lim_{\cU}f'_i$ (as in the proof of Proposition \ref{prop:DLmaps}).

At this point, we do not necessarily know that a given $f'_i$ maps to $C$. However, by assumption there is an open set $U\seq [0,1]^\omega$ containing $C$ and a retraction $r\colon U\to C$. Let $X\seq I$ be the set of $i\in I$ such that $f'_i$ maps to $U$.

\medskip

\noindent\textit{Claim.} $X$ is in $\cU$.

\noindent\textit{Proof.} We view an element $w\in [0,1]^\omega$ as a sequence $(w_k)_{k<\omega}$. Since $C$ is compact and contained in the open set  $U$, we can find finitely many points $w^1,\ldots,w^m\in C$, integers $n_1,\ldots,n_m<\omega$, and real numbers $\epsilon_1,\ldots,\epsilon_m,\delta>0$ such that:
\begin{enumerate}[$(1)$]
\item for any $w\in C$ there is  $t\leq m$ such that $|w_k-w^t_k|\leq\epsilon_t$ for all $k<n_t$, and
\item for any $w\in [0,1]^\omega$ and $t\leq m$, if $|w_k-w^t_k|<\epsilon_t+\delta$ for all $k<n_t$ then $w\in U$.
\end{enumerate}

For each $k<\omega$, define $f^k\colon S^{\cL}_x(\emptyset)\to [0,1]$ to be the $k^{\textnormal{th}}$ coordinate map of $f'$. Then each $f^k$ is a definable $\cL$-predicate in the sense of \cite[Section 9]{BBHU} (see Proposition 8.10 and Theorem 9.9 there). 
Consider the definable $\cL$-predicate
 \[
 \psi(x)\coloneqq \min_{t\leq m}\max_{k<n_t}(|f_k(x)-w^t_k|\dotminus \epsilon_t).
 \]
Since $f$ maps to $C$, we have $\sup_{a\in M^x}\psi(a)=0$ by $(1)$. Therefore, the set $Y\seq I$ of $i\in I$ such that $\sup_{a\in M_i^x}\psi(a)<\delta$ is in $\cU$. By $(2)$, we have $Y\seq X$.\clqed \medskip

For $i\not\in X$, replace $f'_i$ with an arbitrary $\emptyset$-definable map from $M^x_i$ to $U$ (e.g., a constant map). So now $f'_i$ maps to $U$ for all $i\in I$, and since $X\in\cU$  we still have $f=\lim_{\cU}f'_i$. For each $i\in I$, set $f_i=r\circ f'_i$. Then each $f_i$ is an $\emptyset$-definable map from $M^x_i$ to $C$. Moreover, for any $a\in M^x$ and representative $(a_i)_{i\in I}$, we have
\[
\textstyle f(a)=r(f(a))=r(\lim_{\cU}f'_i(a_i))=\lim_{\cU} r(f'_i(a_i))=\lim_{\cU}f_i(a_i),
\]
as desired.
\end{proof}

In personal communication with the first author, James Hanson has proposed an example showing that the conclusion of the previous result can fail if one does not assume $C$ is an absolute neighborhood retract.

\subsection{Formalizing the logic}

Let $\cL$ be a first-order language in either classical or continuous logic. We define $\cL_f$ to be a two-sorted\footnote{See page 9 of \cite{Pibook} and/or the end of Section 1.1 in \cite{TeZi} for  discussion of multi-sorted languages.} continuous language  in sorts $S_1$ and $S_2$ consisting of the following symbols:
\begin{enumerate}[\hspace{10pt}$\ast$]
\item the language $\cL$ relativized to the sort $S_1$ (if $\cL$ is a language in classical logic, we use  trivial moduli of uniform continuity for all symbols in $\cL$);
\item a function symbol $f$  from $S_1$ to $S_2$ of some arity and modulus of uniform continuity (which we  suppress in the notation).
\end{enumerate}
We will write $\cL_{f}$-structures as triples $(M,X,f)$ where $M$ is an $\cL$-structure and $X$ is the universe of the sort $S_2$. Note that in this case $X$ has no further structure other than its metric and the map $f$ from $M$. 

For the rest of this subsection, let $C$ be a  compact metric space with metric $d$.

\begin{definition}
An \textbf{$\cL_f^C$-structure} is an $\cL_{f}$-structure $(M,C,f)$ in which the second sort $S_2$ is interpreted as $C$. 
\end{definition}

\begin{remark}
By restricting to the case that $C$ is a metric space, the above notion  fits properly into \cite{BBHU}, which  is by now the most widely accepted form of ``continuous logic". That being said, other formalisms of continuous logic exist throughout the literature. For example, when $\cL$ is classical,  $\cL_f^C$-structures are a special case of the objects studied by the third author and Chavarria in \cite{ChPi} (where $C$ is allowed to be any compact Hausdorff space).
\end{remark}

\begin{remark}\label{rem:LCfUP}
Suppose $(M_i,C,f_i)_{i\in I}$ is a collection of $\cL_f^C$-structures, and let $\cU$ be an ultrafilter on $I$. Then  the ultraproduct $\prod_{\cU}(M_i,C,f_i)$ is canonically isomorphic to an $\cL_f^C$-structure in the following way. First, we have a well-defined map $\lim_\cU f_i\colon\prod_{\cU}M_i\to C$ such that $(\lim_\cU f_i)((x_i)_{\cU})=\lim_{\cU} f_i(x_i)$. So this yields an $\cL_f^C$-structure $(\prod_{\cU}M_i,C,\lim_\cU f_i)$. Now let $\prod_{\cU}f_i$ denote the map from $\prod_{\cU}M_i$ to $C^{\cU}$ sending $(x_i)_{\cU}$ to $(f_i(x_i))_{\cU}$. Then we have 
\[
\textstyle\prod_{\cU}(M_i,C,f_i)=(\prod_{\cU}M_i,C^{\cU},\prod_{\cU}f_i)\cong (\prod_{\cU}M_i,C,\lim_{\cU}f_i),
\]
where here we are using the fact that, since $C$ is compact, the map $\lim_{\cU}\colon C^{\cU}\to C$ is an isomorphism of metric structures (see the Appendix of \cite{GoLo}). Altogether, we can view an ultraproduct of $\cL_f^C$-structures as an $\cL_f^C$-structure. 
\end{remark}

\begin{remark}\label{rem:MUC}
Let $M$ be an $\cL$-structure and suppose $f\colon M^x\to C$ is an $\emptyset$-definable function. 
We will make a choice of language  $\cL_{f}$ as above, which will allow us to define (non-canonical) expansions of other $\cL$-structures to $\cL_f^C$-structures, perhaps under some additional assumptions.   

\textit{Case 1.} $\cL$ is classical. In this case we assume $\cL$ is countable. Let $\cL_{f}$ be as defined above, with $f$ having  trivial modulus of uniform continuity and arity given by $x$. Now view $f$ as a map from $S^{\cL}_x(\Th(M))$ to $C$. Use Lemma \ref{lem:Kura} to choose a continuous extension $f'\colon S^{\cL}_x(\emptyset)\to C$.  Then given any $\cL$-structure $N$, we can  expand $N$ to an $\cL_f^C$-structure by interpreting $f$ as $f'{\upharpoonright}S^{\cL}_x(\Th(N))$.\footnote{Here we use the symbol $\upharpoonright$ to denote function restriction.}

\textit{Case 2.} $\cL$ is continuous. In this case, we assume $C$ is an absolute neighborhood retract. Let $\rho$ be a metric on $[0,1]^\omega$ compatible with the product topology. View $C$ as topologically embedded in $[0,1]^\omega$. Fix an open set $U\seq [0,1]^\omega$ containing $C$ and a retraction $r\colon U\to C$. View $f$ as a map from $S^{\cL}_x(\Th(M))$ to $C$. Use Tietze Extension to choose a  continuous extension $f'\colon S^{\cL}_x(\emptyset)\to [0,1]^\omega$. 

We now define a function $\Delta\colon\R^+\to\R^+$, which will be used as a modulus of uniform continuity for $f$ as a symbol in $\cL_{f}$. Fix $\epsilon>0$. First choose $\delta>0$ such that if $x,y\in U$ and $\rho(x,y)<\delta$ then $d(r(x),r(y))<\epsilon$. Then choose $\Delta(\epsilon)>0$ such that if $p,q\in S^{\cL}_x(\emptyset)$ and $d_{\cL}(p,q)<\Delta(\epsilon)$ then $\rho(f'(p),f'(q))<\delta$. Altogether, we have that for any $p,q\in S^{\cL}_x(\emptyset)$, if $d_{\cL}(p,q)<\Delta(\epsilon)$ and $f'(p),f'(q)\in U$ then $d(r(f'(p)),r(f'(q)))<\epsilon$.  Let $\cL_{f}$ be defined as above, using $\Delta$ as a modulus of uniform continuity for $f$. 

Finally, we say that an $\cL$-structure $N$ is \emph{$f$-coherent} if $f'{\upharpoonright}S^{\cL}_x(\Th(N))$ maps to $U$. By the above construction, any $f$-coherent $\cL$-structure $N$ can be expanded to an $\cL_f^C$-structure $(N,C,f)$ by interpreting $f$ as $r\circ (f'{\upharpoonright}S^{\cL}_x(\Th(N)))$. 
\end{remark}

For the final result of this section, let $\cL$ and $C$ be as above. In the case that $\cL$ is classical, assume also that $\cL$ is countable. In the case that $\cL$ is continuous, assume also that $C$ is an absolute neighborhood retract.

\begin{corollary}\label{cor:mainmaps}
Let $M=\prod_{\cU}M_i$ be an ultraproduct of a family $\{M_i:i\in I\}$ of $\cL$-structures. Suppose $f\colon M^x\to C$ is $\emptyset$-definable, and let $\cL_f$ be as constructed in Remark \ref{rem:MUC}. Then for each $i\in I$, there is an expansion $(M_i,C,f_i)$ of $M_i$ to an $\cL_f^C$-structure so that $(M,C,f)=\prod_{\cU}(M_i,C,f_i)$. Moreover, each $f_i$ is $\emptyset$-definable as a map on the $\cL$-structure $M_i$. 
\end{corollary}
\begin{proof}
For each $i\in I$, construct an $\emptyset$-definable map $f_i\colon M_i^x\to C$ as in the proof of Proposition \ref{prop:DLmaps} in the classical case, or Proposition \ref{prop:CLmaps} in the continuous case, while working in the context of Remark \ref{rem:MUC}. Then by construction (and Remark \ref{rem:MUC}), each $(M_i,C,f_i)$ is a well-defined $\cL_f^C$-structure, and we have $f=\lim_{\cU}f_i$. So $(M,C,f)=\prod_{\cU}(M_i,C,f_i)$ (as explained in  Remark \ref{rem:LCfUP}).
\end{proof}

\subsection{Homomorphisms to compact Lie groups}

We now prove the main results of this section. Throughout this subsection, we let $\cL$ be a classical or continuous language expanding the language of groups. Recall that an ultrafilter $\cU$ on a set $I$ is \textit{countably incomplete} if it contains a countable collection of sets whose common intersection is empty (for example, any nonprincipal ultrafilter on a countable set is countably incomplete).

\begin{theorem}\label{thm:mainLKt}
Suppose $G=\prod_{\cU}G_i$ is an ultraproduct of a family $\{G_i:i\in I\}$ of $\cL$-structures expanding amenable groups, and assume $\cU$ is countably incomplete. Let $\tau\colon G\to K$ be a definable homomorphism where $K$ is a compact Lie group. Then there are homomorphisms $\tau_i\colon G_i\to K$, for  $i\in I$, such that $\tau=\lim_{\cU}\tau_i$.
\end{theorem}
\begin{proof}
Since $K$ is second countable, $f$ is definable in some countable sublanguage of $\cL$ expanded by countably many constants (as in Remark \ref{rem:DLmaps}$(2)$). So without loss of generality, we may assume $\cL$ is countable and $f$ is $\emptyset$-definable. Note  that the conclusion of the theorem does not depend on the choice of metric on $K$. So we may  assume the metric  satisfies the statement of Theorem \ref{thm:Kaz}. We also recall that $K$ is an absolute neighborhood retract (see \cite[Corollary A.9]{Hatcher}). Let $\cL_{\tau}$ be the continuous language defined in Remark \ref{rem:MUC}, after starting with $\tau\colon G\to K$ as the initial definable function. By Corollary \ref{cor:mainmaps}, we can expand each $G_i$ to an $\cL_\tau^K$-structure $(G_i,K,\tau'_i)$ so that $(G,K,\tau)=\prod_{\cU}(G_i,K,\tau'_i)$. 

Consider the $\cL_{\tau}$-sentence $\psi=\sup_{xy}d(\tau(xy),\tau(x)\tau(y))$. Then $\psi^G=0$. Let $\epsilon_K>0$ be as in Theorem \ref{thm:Kaz}. Given $n>0$, let $I_n=\{i\in I:\psi^{G_i}\leq \epsilon_n\}$ where $\epsilon_n=\min\{\frac{1}{n},\epsilon_K\}$. Then $I_n\in\cU$ for all $n>0$. Moreover, for each $n>0$ and $i\in I_n$, $\tau'_i\colon G_i\to K$ is an $\epsilon_n$-approximate homorphism, and so by Theorem \ref{thm:Kaz} there is a homomorphism $\tau_{n,i}\colon G_i\to K$ such that $d(\tau'_i(x),\tau_{n,i}(x))\leq 2\epsilon_n\leq \frac{2}{n}$ for all $x\in G_i$. We also define $\tau_{0,i}\colon G_i\to K$ to be the trivial homomorphism for each $i\in I$. 

Now, by assumption on $\cU$, we may fix a countable sequence $(A_n)_{n=1}^\infty$ of sets in $\cU$ such that $\bigcap_{n\in S}A_n=\emptyset$ for any infinite $S\seq\Z^+$. For each $i\in I$, let $n_i\geq 1$ be maximal such that $i\in A_{n_i}$. Define $k_i$ to be  the maximal $k\in \{1,\ldots,n_i\}$ such that $i\in I_k$ if such a $k$ exists, and $k_i=0$ otherwise. Finally, set $\tau_i=\tau_{k_i,i}$. We will show $\tau=\lim_{\cU}\tau_i$. 

Fix $a\in G$ and choose a representative $(a_i)_{i\in I}$. To prove $\tau(a)=\lim_{\cU}\tau_i(a_i)$, we fix $n>0$ and show that the set 
\[
X_n\coloneqq \{i\in I:d(\tau(a),\tau_i(a_i))<{\textstyle\frac{3}{n}}\}
\]
is in $\cU$.  Let $Y_n=\{i\in I:d(\tau(a),\tau'_i(a_i))<\frac{1}{n}\}$. Then $Y_n\in\cU$, and so $Z_n\coloneqq A_n\cap I_n\cap Y_n\in\cU$. We prove $Z_n\seq X_n$, hence $X_n\in\cU$. Fix $i\in Z_n$. Then $i\in A_n$ so $n_i\geq n$. Since $i\in I_n$, it follows that $k_i\geq n>0$. Therefore $d(\tau'_i(a_i),\tau_i(a_i))=d(\tau'_i(a_i),\tau_{k_i,i}(a_i))\leq\frac{2}{k_i}\leq \frac{2}{n}$. Also, $d(\tau(a),\tau'_i(a_i))<\frac{1}{n}$ since $i\in Y_n$. By the triangle inequality, we have $i\in X_n$, as desired. 
\end{proof}

In the previous proof we lose definability of $\tau_i$ when applying Theorem \ref{thm:Kaz}. In general, we may also lose uniform continuity. However, in applications to additive and multiplicative combinatorics it is typical to work in the case that $G$ is a classical structure, or a continuous structure with a discrete metric. In these cases, we  recover the  analogue of Corollary \ref{cor:mainmaps} (but without definability) as follows.

Suppose $G=\prod_{\cU}G_i$ is an ultraproduct of a family $\{G_i:i\in I\}$ of $\cL$-structures expanding amenable groups, and assume $\cU$ is countably incomplete. Assume further that if $\cL$ is continuous then each $G_i$ is given the discrete metric. Let $\tau\colon G\to K$ be a definable homomorphism where $K$ is a compact Lie group. Define the language $\cL_{\tau}$ as in the proof of the previous theorem, using a trivial modulus of uniform continuity for $\tau$. Recall from the discussion before Remark \ref{rem:MUC} that we may view ultraproducts of $\cL_\tau^K$-structures as $\cL_\tau^K$-structures in a canonical way. 

\begin{corollary}\label{cor:mainLKt}
Let $G$ be an ultraproduct of $\cL$-structures expanding amenable groups satisfying the assumptions above, and let $\tau\colon G\to K$ be a definable homomorphism to a compact Lie group $K$. Then each $G_i$ can be expanded to an $\cL_\tau^K$-structure $(G_i,K,\tau_i)$, with $\tau_i$ a homomorphism, so that $(G,K,\tau)=\prod_{\cU}(G_i,K,\tau_i)$.
\end{corollary}

The previous results are phrased for groups \emph{equal} to ultraproducts. So now let $\cC$ be some class of $\cL$-structures expanding amenable groups with discrete metrics. An $\cL$-structure $G$ is \emph{pseudo-$\cC$} if for every $\cL$-sentence $\psi$ and $\epsilon>0$, there is some $H\in\cC$ such that $|\psi^G-\psi^H|<\epsilon$. Note that any pseudo-$\cC$ $\cL$-structure $G$ has a discrete metric, and so for any definable map $f$ from $G$ to a compact metric space $C$, we have a canonical $\cL_f^C$-structure $(G,f,C)$ using the trivial modulus of uniform continuity for $f$.

\begin{corollary}
Let $\cC$ be as described above.
Suppose $G$ is a pseudo-$\cC$ $\cL$-structure, and $\tau\colon G\to K$ is a definable homomorphism to a compact Lie group $K$. Then the $\cL_\tau^K$-structure $(G,K,\tau)$ is elementarily equivalent to an ultraproduct $\prod_{\cU}(G_i,K,\tau_i)$ of $\cL_\tau^K$-structures where each $G_i$ is in $\cC$ and each $\tau_i$ is a homomorphism.
\end{corollary}
\begin{proof}

By assumption, $G\equiv\prod_{\cU} G_i$ for some sequence $(G_i)_{i\in I}$ of structures in $\cC$ and  ultrafilter $\cU$ on $I$. This is proved for continuous logic in \cite[Lemma 2.4]{GoLo}. In the proof, one  sets $I=\mathcal{P}_{\text{fin}}(\Th(G))\times \Z^+$ and then lets $\cU$ be any ultrafilter containing the sets $\{(x,\ell)\in I:s\seq x,~k\leq\ell\}$ for all $(s,k)\in I$. In particular, we may assume $\cU$ is countably incomplete (even regular; see \cite[Exercise 38.5]{Jech}). 

Now, by the Keisler-Shelah Theorem \cite[Theorem 5.7]{BBHU}, there is a set $J$ and an ultrafilter $\cV$ on $J$ such that $G^{\cV}\cong (\prod_{\cU}G_i)^{\cV}$. In order to write the latter structure as a single ultraproduct, we define $G_{i,j}=G_i$ for $(i,j)\in I\times J$. Then a routine  exercise shows that $(\prod_{\cU}G_i)^{\cV}$ is isomorphic to $H:=\prod_{\cU\otimes\cV}G_{i,j}$. So we have an isomorphism $F\colon G^{\cV}\to H$. Let $\tau^{\cV}\colon G^{\cV}\to K$ denote the ultralimit map $\lim_{\cV}\tau$. Then another routine exercise shows that $\tau^{\cV}$ is definable (over $G$). Therefore $\sigma=\tau^{\cV}\circ F\inv$ is a definable homomorphism from $H$ to $K$ and, moreover, $F$ induces an $\cL_\tau^K$-structure isomorphism from $(G^{\cV},K,\tau^{\cV})$ to $(H,K,\sigma)$. Note also that $(G,K,\tau)$ and $(G^{\cV},K,\tau^{\cV})$ are elementarily equivalent as $\cL_\tau^K$-structures. Altogether, $(G,K,\tau)$ and $(H,K,\sigma)$ are elementarily equivalent. Finally, apply Corollary \ref{cor:mainLKt} to $(H,K,\sigma)$, noting that $\cU\otimes\cV$ is also countably incomplete.
\end{proof}

\section{Unitary Bohr Neighborhoods and  Amenable Groups}\label{sec:unitary}

Throughout this section, we will use the following standard notation for product sets in groups. Given a group $G$ and $A,B\seq G$, let $AB=\{ab:a\in A,~b\in B\}$. The sets $A^n$ for $n\geq 1$ are then defined inductively as $A^1=A$ and $A^{n+1}=A^nA$. Finally, $A\inv=\{a\inv:a\in A\}$ and, for $n<0$, $A^n=(A\inv)^{\nv n}$.

\subsection{Preliminaries on Bohr neighborhoods} 

We recall the following notation used in \cite{CPT} for Bohr neighborhoods obtained from maps to arbitrary metric groups.

\begin{definition}\label{def:Bohr}
Let $K$ be a  metric group. Given a group $G$ and a real number $\delta>0$, a \textbf{$(\delta,K)$-Bohr neighborhood in $G$} is a set of the form $\tau\inv(U)$ where $\tau\colon G\to K$ is a homomorphism and $U$ is the open identity neighborhood in $K$ of radius $\delta$.
\end{definition}

Bohr neighborhoods are often used in additive combinatorics as approximations to subgroups. We note some basic properties along these lines. 

\begin{remark}
Let $K$ be a metric group.  Suppose $B$ is a $(\delta,K)$-Bohr neighborhood in a group $G$, witnessed by  $\tau\colon G\to K$. Then $B=B\inv$, $1\in B$, and $B$ is ``normal" in the sense that $gBg\inv=B$ for any $g\in G$. Moreover, $B^2$ is contained in the $(2\delta,K)$-Bohr neighborhood defined from $\tau$. 
\end{remark}

One usually sees Bohr neighborhoods defined in the setting where $K$ is compact. We will focus mainly on the case of unitary groups and torus groups. Let $\U(n)$ denote the unitary group of degree $n$ with metric induced by the standard matrix operator norm on $\GL(n)$ (this is also called  the ``spectral  norm"). Let $\T(n)$ denote the maximal torus in $\U(n)$ consisting of all diagonal matrices. 

\begin{remark}\label{rem:metric}
The metric on $\U(n)$ restricts to $\T(n)$ as the product of the complex distance metric on $S^1$ (in $\C$). This is slightly different than  sources such as \cite{CoBogo,CPT,GreenSLAG}, which use   the arclength metric on $S^1$. So in particular,  a $(\delta,\T(n))$-Bohr neighborhood here corresponds to a  ``$(\delta',n)$-Bohr neighborhood" in \cite{CoBogo,CPT}, where $\delta'$ depends uniformly on $\delta$ to account for the change in metric on $S^1$. 
\end{remark}

By definition, any $(\delta,\T(n))$-Bohr neighborhood is also  a $(\delta,\U(n))$-Bohr neighborhood. For torsion groups or abelian groups, one obtains the following converse statements. 

\begin{proposition}\label{prop:UtoT}
Let $B$ be a $(\delta,\U(n))$-Bohr neighborhood in a group $G$.
\begin{enumerate}[$(a)$]
\item If $G$ is abelian then $B$ is a $(\delta,\T(n))$-Bohr neighborhood.
\item If $G$ is a torsion group then there is a normal subgroup $H\leq G$ of index $O_n(1)$ such that $B\cap H$ is a $(\delta,\T(n))$-Bohr neighborhood in $H$.
\end{enumerate}
\end{proposition}
\begin{proof}
Fix a homomorphism $\tau\colon G\to \U(n)$ witnessing that $B$ is a $(\delta,\U(n))$-Bohr neighborhood.

Part $(a)$. Assume $G$ is abelian. Then the image of $\tau$ is an abelian subgroup of $\U(n)$, and hence contained in some conjugate of $\T(n)$ (see \cite[Theorem 2]{Newman}).  So after replacing $\tau$ by a conjugate, we may  assume $\tau$ maps to $\T(n)$.

Part $(b)$. Assume $G$ is a torsion group. By Fact \ref{fact:JST}, there is a normal abelian subgroup $A\leq \tau(G)$ of index $d\leq O_n(1)$ in $\tau(G)$. As in part $(a)$, we  may assume  $A\leq T(n)$. Let $H=\tau\inv(A)$. Then $H$ is a normal subgroup of $G$ of index $d$. Let $\tau_0\colon H\to T(n)$ be the restriction of $\tau$ to $H$. Then $B\cap H=\tau_0\inv(U)$, where $U$ is the open identity neighborhood of radius $\delta$ in $\T(n)$. 
\end{proof}

Next we show that Bohr neighborhoods are ``large". For finite groups, a lower bound on the cardinality of a Bohr neighborhood can be obtained from  an averaging argument. See \cite[Lemma 4.1]{GreenSLAG} or \cite[Lemma 4.20]{TaoVu}, both of which deal with abelian groups; the proof is rewritten for arbitrary finite groups in \cite[Proposition 4.5]{CPT}. The same method would work to bound the measure of a Bohr neighborhood in an  amenable group. Instead we will give here a short elementary proof valid for any group, where ``large" is formulated using genericity. Given a group $G$, we say that $A\seq G$ is \textbf{$n$-generic (in $G$)} if $G$ can be covered by $n$ left translates of $A$. Note that if $G$ is finite this implies $|A|\geq |G|/n$. More generally, if $G$ is amenable then this implies $\mu(A)\geq 1/n$ for \emph{any} left-invariant measure $\mu$ on $G$.

\begin{lemma}\label{lem:bigBohr}
Let $K$ be a group and fix $U\seq K$. 
Suppose there is a subset $V\seq K$ such that $V\inv V\seq U$ and $V$ is $n$-generic in $K$ for some $n\geq 1$. Then for any group $G$ and any homomorphism $\tau\colon G\to K$, the set $\tau\inv(U)$ is $n$-generic in $G$.
\end{lemma}
\begin{proof}
Fix $\tau\colon G\to K$ and let $B=\tau\inv (U)$. Fix a finite set $E\seq K$ of size  $n$ such that $K=EV$. Let $E_0$ be the set of $a\in E$ such that $\tau(G)\cap aV\neq\emptyset$. For each $a\in E_0$, fix some $v_a\in V$ such that $av_a\in\tau(G)$, and fix some $g_a\in G$ such that $\tau(g_a)=av_a$. Set $F=\{g_a:a\in E_0\}$. Then $|F|\leq n$, and we show that $G=FB$. 

Fix $x\in G$. Choose $a\in E$ such that $\tau(x)\in aV$. Then $a\in E_0$, and we have 
\[
\tau(x)\in av_av_a\inv V=\tau(g_a)v_a\inv V\seq \tau(g_a)V\inv V\seq \tau(g_a)U.
\]
Therefore $\tau(g_a\inv x)\in U$, i.e., $g_a\inv x\in B$, i.e., $x\in g_aB\seq FB$.
\end{proof}

In the context of the previous proposition, if $K$ is a topological group and $U$ is an identity neighborhood, then one can always find some identity neighborhood $V$ satisfying $V\inv V\seq U$. If $K$ is also compact, then such a $V$ is always $n$-generic for some $n\geq 1$. This can be made  quantitative in metric groups using covering numbers. In particular, given a compact metric group $K$, let  $C_{K,\epsilon}$ be the \emph{$({<}\epsilon)$-covering number for $K$}, i.e., the minimum size of an $({<}\epsilon)$-net for $K$. Let $U_\epsilon$ denote the open identity neighborhood in $K$ of radius $\epsilon$. Then a subset $E\seq K$ is an $({<}\epsilon)$-net for $K$ if and only if $K=EU_\epsilon$. Therefore $C_{K,\epsilon}$ is the least integer $n$ such that $U_\epsilon$ is $n$-generic in $K$. Moreover, by the triangle inequality we have $U_\epsilon\inv U_\epsilon\seq U_{2\epsilon}$. Putting these remarks in the context of Lemma \ref{lem:bigBohr}, we obtain the following conclusion.

\begin{corollary}
Let $K$ be a compact metric group. Then for any group $G$, any $(\delta,K)$-Bohr neighborhood in $G$ is $C_{K,\delta/2}$-generic in $G$.
\end{corollary}

The following bound on covering numbers in unitary groups is proved  in \cite{Szarek}. 

\begin{fact}\label{fact:covU}
$C_{U(n),\epsilon}\leq (c_0/\epsilon)^{n^2}$ for some absolute constant $c_0>0$. 
\end{fact}

\begin{corollary}\label{cor:Bohrbound}
Let $G$ be a group. Then any $(\delta,\U(n))$-Bohr set in $G$ is $(c/\delta)^{n^2}$-generic in $G$ for some absolute constant $c>0$.\end{corollary}

\textbf{For the rest of this section, $c$  denotes the  constant in the previous corollary.} (So $c=2c_0$ where $c_0$ is from Fact \ref{fact:covU}.)

\subsection{Bogolyubov's lemma for amenable groups} We now come to the main result of this section. 

\begin{theorem}\label{thm:amenBogo}
Fix $\alpha>0$ and a function $\epsilon\colon\R^+\times\Z^+\to \R^+$. 
Let $G$ be an amenable group with a left-invariant measure $\mu$. Suppose $A\seq G$ is such that $\mu(A)\geq \alpha$. Then there is a $(\delta,\U(n))$-Bohr neighborhood $B\seq G$, with $\delta\inv,n\leq O_{\alpha,\epsilon}(1)$, such that:
\begin{enumerate}[$(i)$]
\item  $B\seq (AA\inv)^2$, 
\item $AAA\inv$ contains a translate of $B$, and 
\item $\mu(B\backslash AA\inv)<\epsilon(\delta,n)\mu(B)$. 
\end{enumerate}
\end{theorem}

\begin{remark}\label{rem:amenBogo}
The main model-theoretic ingredient in Theorem \ref{thm:amenBogo} is a variation of the second author's stabilizer theorem  \cite{HruAG} proved by Montenegro, Onshuus, and Simon \cite{MOS}.  The original version from \cite{HruAG} would apply if we made the stronger assumption of bi-invariance of $\mu$.
 The variation from \cite{MOS} formulates certain assumptions under which left-invariance suffices to obtain  conditions $(i)$ and $(iii)$ (in the full generality of ``S1-ideals"). Assuming bi-invariance, the authors of \cite{MOS} also obtain further aspects that would yield  $(ii)$.
 In our case we will be able to obtain $(ii)$ due to the fact that we are working with a particularly nice S1-ideal, namely the null sets of a probability  measure. 
\end{remark}

\begin{proof}[\textnormal{\textbf{Proof of Theorem \ref{thm:amenBogo}}}]
Suppose not. Then  we obtain a sequence $(G_s,A_s,\mu_s)_{s\geq 1}$ such that for all $s\geq 1$, $G_s$ is an amenable group with a left-invariant measure $\mu_s$, $A_s\seq G_s$ with $\mu_s(A_s)\geq\alpha$, and there is no $(\delta,\U(n))$-Bohr neighborhood $B$ in $G_s$ satisfying $(i)$, $(ii)$, and $(iii)$  with $\delta\inv, n\leq s$. Let $\cU$ be a nonprincipal ultrafilter on $\Z^+$. Set $G=\prod_{\cU}G_s$ with the full internal language. Let $A=\prod_{\cU}A_s$ and let $\mu=\lim_{\cU}\mu_s$. By standard arguments, $\mu$ is definable over $\emptyset$ with respect to a suitable countable sublanguage. So going forward we may work in a countable language $\cL$ (which still contains the group language and a predicate for $A$).

 By construction, $\mu(A)\geq\alpha>0$.  Now suppose $\widetilde{G}\succ G$ is sufficiently saturated. Given a definable set $X\seq G$, we let $\widetilde{X}$ denote $X(\widetilde{G})$. Let $\widetilde{\mu}$ be the canonical $\emptyset$-definable extension of $\mu$ to a Keisler measure over $\widetilde{G}$. The next claim condenses the part of the proof requiring the stabilizer theorem of \cite{MOS}. We assume  familiarity with the relevant notions. In our setting, the structures $\widetilde{G}$ and $G$ play the role of $G$ and $M$ in \cite{MOS} (respectively). Thus the type space $S_G(M)$ in \cite{MOS} is for us  $S_1(G)$, i.e., the Stone space of complete types over the Boolean algebra of $G$-definable subsets of $\widetilde{G}$.

 \medskip
 
 \noindent\emph{Claim 1.} There is a $G$-type-definable bounded-index normal subgroup $\Gamma\leq \widetilde{G}$ such that $\Gamma\seq (\widetilde{A}\widetilde{A}\inv)^2$, $\widetilde{A}\widetilde{A}\widetilde{A}\inv$ contains a coset of $\Gamma$, and   $\Gamma\backslash \widetilde{A}\widetilde{A}\inv$ is contained in $\widetilde{Z}$ for some $G$-definable set $Z$ with $\mu(Z)=0$.
 
 \noindent\emph{Proof.} We use \emph{wide} to mean $\widetilde{\mu}$-wide. Since $\widetilde{\mu}$ is a $G$-invariant Keisler measure on $\widetilde{G}$, the ideal of non-wide definable subsets of $\widetilde{G}$ is $G$-invariant and S1 on $\widetilde{G}$ (see the discussion at the start of \cite[Section 2.1]{MOS}). Since $\widetilde{A}$ is $G$-definable and wide, there is a wide type $p\in S_1(G)$ containing $\widetilde{A}$. Let $St(p)=\{g\in\widetilde{G}:gp\cap p\text{ is wide}\}$, and set $\Gamma=St(p)^2$. By \cite[Theorem 2.12]{MOS} (using $X=\widetilde{G}$ to satisfy assumption (B1)), $\Gamma$ is a wide $G$-type-definable connected subgroup of $\widetilde{G}$. Moreover,  $\Gamma=(pp\inv)^2$ and $\Gamma\backslash St(p)$ is contained in a small union of non-wide $G$-definable sets. Since $\widetilde{\mu}(\widetilde{G})$ is finite, it follows that $\Gamma$ has bounded index in $\widetilde{G}$. Then, as $\Gamma$ is type-definable over $G$ and connected, it follows that $\Gamma$ is normal in $\widetilde{G}$. 
 
 Note that $\Gamma=(pp\inv)^2\seq (\widetilde{A}\widetilde{A}\inv)^2$. Also, since $St(p)\seq pp\inv\seq \widetilde{A}\widetilde{A}\inv$, we have $\Gamma\backslash \widetilde{A}\widetilde{A}\inv\seq \Gamma\backslash St(p)$. Since $\Gamma\backslash \widetilde{A}\widetilde{A}\inv$ is type-definable over $G$, we conclude that  $\Gamma\backslash \widetilde{A}\widetilde{A}\inv$ is contained in a non-wide (i.e., $\widetilde{\mu}$-null) $G$-definable subset of $\widetilde{G}$. It remains to show that $\widetilde{A}\widetilde{A}\widetilde{A}\inv$ contains a coset of $\Gamma$. First, since $\widetilde{A}$ is wide and definable, and $\Gamma$ has bounded index, there must be some coset $g\Gamma$ such that $\widetilde{A}\cap g\Gamma$ is wide. We show $g\Gamma\seq \widetilde{A}\widetilde{A}\widetilde{A}\inv$. Fix $a\in \Gamma$. Then $a\inv g\inv \widetilde{A}\cap \Gamma$ is a wide subset of $\Gamma$. By our previous conclusion for $\Gamma\backslash \widetilde{A}\widetilde{A}\inv$, it follows that $a\inv g\inv \widetilde{A}\cap \Gamma$ must intersect $\widetilde{A}\widetilde{A}\inv$. So $a\inv g\inv \widetilde{A}\cap \widetilde{A}\widetilde{A}\inv\neq\emptyset$, i.e., $ga\in \widetilde{A}\widetilde{A}\widetilde{A}\inv$. \clqed\medskip

 Let $\Gamma\leq\widetilde{G}$ and $Z\seq G$ be as in Claim 1. Then $\widetilde{G}\!/\Gamma$ is a compact group under the logic topology (see \cite{PilCLG}).
One can then use Peter-Weyl to replace $\Gamma$ by a subgroup of $\widetilde{G}$ satisfying the same conclusions of the Claim, but with $\widetilde{G}\!/\Gamma$ a compact Lie group (e.g., this follows from the proof of \cite[Lemma 2.9]{CPT} together with the fact that every coset of $\Gamma$ is type-definable over $G$).

Let $K=\widetilde{G}\!/\Gamma$ and let $\widetilde{\tau}\colon \widetilde{G}\to K$ be the quotient homomorphism. Note that $\widetilde{\tau}$ is definable over $G$, and  $\tau\coloneqq\widetilde{\tau}|_G$ is a definable compactification of $G$.    \medskip

\noindent\emph{Claim 2.} There is some $\gamma>0$ and some $g\in G$ satisfying the following properties.
\begin{enumerate}[$(i)$]
\item For any $x\in G$, if $d(\tau(x),1)<\gamma$ then $x\in (AA\inv)^2\cap ((AA\inv)\cup Z)$. 
\item For any $x\in G$, if $d(\tau(x),\tau(g))<\gamma$ then $x\in AAA\inv$.
\end{enumerate} 

\noindent\emph{Proof.}
To ease notation, set $D=(AA\inv)^2\cap ((AA\inv)\cup Z$ and $E=AAA\inv$.
By Claim 1, we have $\Gamma\seq \widetilde{D}$ and $a\Gamma\seq \widetilde{E}$ for some $a\in \widetilde{G}$. Given $\gamma>0$, set $B_\gamma=\{g\in\widetilde{G}:d(\widetilde{\tau}(g),1)\leq\gamma\}$. Then $\Gamma=\bigcap_{\gamma>0}B_\gamma$ and, since $\widetilde{\tau}$ is definable, each $B_\gamma$ is type-definable. By compactness, there is some $\gamma>0$ such that $B_\gamma\seq \widetilde{D}$ and $aB_\gamma\seq \widetilde{E}$. In other words, for any $x\in\widetilde{G}$ if $d(\widetilde{\tau}(x),1)<\gamma$ then $x\in \widetilde{D}$ and if $d(\widetilde{\tau}(x),\widetilde{\tau}(a))<\gamma$ then $x\in \widetilde{E}$. By the triangle inequality and density of $\widetilde{\tau}(G)$ in $K$, we can shrink $\gamma$ and obtain the same conclusion but with $a$ replaced by some $g\in G$. The claim now follows.\clqed\medskip

Fix $\gamma>0$ and $g\in G$ as in Claim 2. Since $K$ is a compact Lie group, we may assume it is a closed subgroup of $\U(n)$ for some $n$, and hence view $\tau$ as a definable homomorphism from $G$ to $\U(n)$. By Corollary \ref{cor:mainLKt}, there are homomorphisms $\tau_s\colon G_s\to \U(n)$ for each $s\geq 1$ so that $(G,\U(n),\tau)=\prod_{\cU} (G_s,\U(n),\tau_s)$ as $\cL_\tau^{\U(n)}$-structures.  Choose a representative $(g_s)_{s\geq 1}$ for $g$, and let $Z=\prod_\cU Z_s$ where each $Z_s\seq G_s$. Set $\delta=\gamma/2$ and let $I$ be the set of $s\geq 1$ satisfying the following properties:
\begin{enumerate}[$(1)$]
\item For any $x\in G_s$, if $d(\tau_s(x),1)<\delta$ then $x\in (A_sA_s\inv)^2\cap ((A_s A_s\inv)\cup Y_s)$.
\item For any $x\in G_s$, if $d(\tau_s(x),\tau_s(g_s))<\delta$ then $x\in A_sA_s A_s\inv$. 
\item $\mu_s(Z_s)<\epsilon(\delta,n)(c/\delta)^{n^2}$.
\end{enumerate}
Then $I\in \cU$ by the Claim and {\L}o\'{s}'s Theorem, and since $\delta<\gamma$ and $\mu(Z)=0$. 

Since $\cU$ is nonprincipal, we may choose some $s\in I$ such that $\delta\inv, n\leq s$. Let $B=\{x\in G_s:d(\tau_s(x),1)<\delta\}$. Then $B$ is a $(\delta,\U(n))$-Bohr neighborhood in $G_s$. By $(1)$, we have $B\seq (A_s A_s\inv)^2$ and $B\backslash A_s A_s\inv\seq Z_s$. By $(2)$ and translation invariance of $d$, we have $g_sB\seq A_sA_s A_s\inv$. Finally, by $(3)$ and Corollary \ref{cor:Bohrbound}, we have
\[
\mu_s(B\backslash A_s A_s\inv)\leq\mu(Z_s)<\epsilon(\delta,n)(c/\delta)^{n^2}\leq \epsilon(\delta,n)\mu_s(B).
\]
Altogether, this contradicts the choice of $(G_s,A_s,\mu_s)$. 
\end{proof}

\begin{remark}\label{rem:covering}
We point out that since Theorem \ref{thm:amenBogo} is formulated with an arbitrary \emph{function} $\epsilon$, conditions $(i)$ and $(ii)$ can be obtained directly from $(iii)$.  In particular, fix $\alpha>0$ and $\epsilon\colon \R^+\times\Z^+\to\R^+$. Without loss of generality, assume $\alpha\leq 1/2$. Define $\epsilon^*\colon\R^+\times\Z^+\to\R^+$ so that
\[
\epsilon^*(\delta,n)=\min\{\epsilon(\delta/2,n),\alpha\} (\delta/2c)^{n^2}.
\]
Now let $G$ be an amenable group with a left-invariant measure $\mu$, and fix $A\seq G$ with $\mu(A)\geq\alpha$. Suppose we have a $(\delta,\U(n))$-Bohr set $B_0\seq G$ with $\mu(B_0\backslash AA\inv)<\epsilon^*(\delta,n)$. Let $B$ be the $(\delta/2,\U(n))$-Bohr set in $G$ defined using the same homomorphism to $\U(n)$ that yields $B_0$. So $B^2\seq B_0$.  We claim that  $B$ satisfies conditions $(i)$, $(ii)$, and $(iii)$ with respect to $\epsilon$. For $(iii)$, note that
\[
\mu(B\backslash A A\inv)\leq \mu(B_0\backslash A A\inv)<\epsilon^*(\delta,n)\leq \epsilon(\delta/2,n)(\delta/2c)^{n^2}\leq \epsilon(\delta/2,n)\mu(B).
\]
Conditions $(i)$ and $(ii)$ can be obtained from the following general statement.

Suppose $U,V\seq G$ are such that $1\in U=U\inv$ and  $U^2\seq V$. Fix some $W\seq G$.
\begin{enumerate}[$(1)$]
\item If $\mu(V\backslash W)<\frac{1}{2}\mu(U)$ then $U\seq WW\inv$. 
\item If $U$ is $m$-generic and $\mu(V\backslash W)<\frac{\alpha}{m}$, then $AW\inv$ contains a left translate of $U$.
\end{enumerate}
Setting $U=B$ (so $m\leq (c/2\delta)^{n^2}$ in (2)), $V=B_0$, and $W=AA\inv$, these  statements yield $(i)$ and $(ii)$, respectively. We leave the proofs of $(1)$ and $(2)$ to the reader (for $(2)$, the argument is similar to final part of the proof of Claim 1 in Theorem \ref{thm:amenBogo}). 
\end{remark}

Next we derive some corollaries of Theorem \ref{thm:amenBogo}. By a \emph{representation} of a group $G$, we mean a homomorphism from $G$ to $\textnormal{GL}_n(\C)$ for some $n$, called the \emph{dimension} of the representation. A \emph{unitary} representation is a representation mapping to $\U(n)$. As an aside, we note the result of Dixmier \cite{Dixmier} that any uniformly bounded representation of an amenable group can be conjugated to a unitary representation. 

\begin{corollary}\label{cor:Gowers}
 For any $\alpha,\epsilon>0$ there is an integer $d\geq 1$ such that the following holds. Let $G$ be an amenable group with a left-invariant measure $\mu$. Suppose $A\seq G$ is such that $\mu(A)\geq\alpha$, and assume $G$ has no nontrivial unitary representations of dimension less than $d$. Then $G=AAA\inv$ and $\mu(AA\inv )>1-\epsilon$. 
 \end{corollary}
 
 The previous corollary connects to \emph{quasirandom} finite groups, which are defined by Gowers in \cite{GowQRG} using graph-theoretic quasirandomness. Roughly speaking, Gowers shows that a finite group is quasirandom if and only if its nontrivial representations have large dimension  (see \cite[Theorem 4.5]{GowQRG} for a precise statement). It is also shown that if $G$ is a finite group with no nontrivial representations of dimension less than $d$, and  if $A,B,C\seq G$ each have size larger than $|G|/d^{1/3}$, then $G=ABC$.\footnote{See \cite[Corollary 1]{NiPy} for an explanation of how this statement follows from results in \cite{GowQRG}.} So Corollary \ref{cor:Gowers} can be viewed as a non-quantitative extension of this result to arbitrary amenable groups, but only with triple products of the form $AAA\inv $. (Though note that in the context of the corollary, if we also have $B\seq G$ with $\mu(B)\geq\epsilon$ then any left translate of $B$ intersects $AA\inv$, and thus $G=BAA\inv$.)
 
 We can also consider (necessarily infinite) groups with no nontrivial unitary representations of any finite dimension, in which case we have the following conclusion.
 
 \begin{corollary}\label{cor:simple}
 Let $G$ be an amenable group with a left-invariant measure $\mu$, and assume that $G$ has no nontrivial finite-dimensional unitary representations. Then for any $A\seq G$, if $\mu(A)>0$ then $G=AAA\inv$ and $\mu(AA\inv )=1$.
 \end{corollary}

The previous corollary applies to any infinite  simple amenable group, since such groups have no nontrivial finite-dimensional representations by the Tits alternative for linear groups (in characteristic $0$; see \cite[Theorem 1]{Tits}). Examples of such groups include Hall's universal group \cite{Hall} or a finitary alternating group on an infinite set (both of which are locally finite, hence amenable). The existence of infinite finitely-generated simple amenable groups was proved by Juschenko and Monod \cite{JuschMon}.

 As a final corollary, we note that as is usually the case for  results of this kind, if one restricts to groups of some fixed finite exponent, then Bohr neighborhoods can  be replaced by subgroups (see Corollary \ref{cor:boundedexp} for a precise statement). The simplest way to obtain this is to adjust the proof of Theorem \ref{thm:amenBogo} to reflect the fact that any compact Hausdorff torsion group is profinite \cite[Theorem 4.5]{Iltis}. This removes the need to work with Bohr neighborhoods, approximate homomorphisms, Kazhdan's theorem, or continuous logic. Indeed, Corollary \ref{cor:boundedexp} is a relatively straightforward consequence of the stabilizer theorems of \cite{HruAG,MOS} (and the aforementioned  fact from \cite{Iltis}). In light of these remarks, we will give an alternate proof which is  the same in spirit, but argues directly from Theorem \ref{thm:amenBogo} as stated, together with the following fact (well-known for $(\delta,\T(n))$-Bohr neighborhoods).

 \begin{fact}\label{fact:Unexp}
 Fix an integer $r\geq 1$, and let $\gamma_r$ be the complex distance between $e^{2\pi i/r}$ and $1$.
 Suppose $G$ is a group of  exponent $r$ and $B$ is a $(\delta,\U(n))$-Bohr neighborhood in $G$, with $\delta\leq \gamma_r$. Then $B$ is a normal subgroup of $G$.
 \end{fact}
 \begin{proof}
Let $\tau\colon G\to\U(n)$ be a homomorphism such that $B=\tau\inv(U)$, where $U$ is the open identity neighborhood in $\U(n)$ of radius $\delta$. Then $\tau(G)$ is a subgroup of $\U(n)$ of exponent $r$. We claim that $\tau(G)\cap U=\{1\}$, which then will yield $B=\ker\tau$. To establish the claim, it suffices to fix some nontrivial $x\in \U(n)$ with $x^r=1$, and show $d(x,1)\geq\delta$. Since $x$ is (unitarily) diagonalizable, there are $y\in \U(n)$ and $z\in \T(n)$ such that $x=y zy\inv$. By invariance of the metric, we have $d(x,1)=d(z,1)$. Moreover, $z^r=y\inv x^r y=1$. So each diagonal entry in $z$ is a complex $r^{\textnormal{th}}$ root of unity. Thus $d(z,1)\geq\gamma_r\geq \delta$.
 \end{proof}

 \begin{corollary}\label{cor:boundedexp}
 Fix $\alpha>0$, $r\geq 1$, and a function $\epsilon\colon\Z^+\to \R^+$. Let $G$ be an amenable group of exponent $r$, and suppose $\mu$ is a left-invariant measure on $G$. Suppose $A\seq G$ is such that $\mu(A)\geq \alpha$. Then there is a normal subgroup $H\leq G$ of index $m\leq O_{\alpha,r,\epsilon}(1)$ such that $H\seq (AA\inv)^2$, $AAA\inv$ contains a coset of $H$, and $\mu(H\backslash AA\inv)<\epsilon(m)\mu(H)$.
 \end{corollary}
 \begin{proof}
 Let $\gamma_r$ be as in Fact \ref{fact:Unexp}. Define $\epsilon^*\colon\R^+\times\Z^+\to\R^+$ so that
 \[
 \epsilon^*(\delta,n)=\min\{\epsilon(m)/m:m\leq (c/\!\min\{\delta,\gamma_r\})^{n^2}\}.
 \]
  Applying Theorem \ref{thm:amenBogo} with $\alpha$ and $\epsilon^*$, we obtain a $(\delta,\U(n))$-Bohr neighborhood $B$ in $G$, with $\delta\inv, n\leq O_{\alpha,r,\epsilon}(1)$, such that $B\seq (AA\inv)^2$, $AAA\inv$ contains a translate of $B$, and $\mu(B\backslash AA\inv)<\epsilon^*(\delta,n)$. Let $\delta_*=\min\{\delta,\gamma_r\}$, and note that we still have $\delta_*\inv\leq O_{\alpha,r,\epsilon}(1)$.  Let $H$ be the $(\delta_*,\U(n))$-Bohr neighborhood in $G$ defined using the same homomorphism to $\U(n)$ that yields $B$. Then $H$ is a normal subgroup of $G$ by Fact \ref{fact:Unexp}. Clearly $H\seq B$, so $H\seq (AA\inv)^2$ and $AAA\inv$ contains a coset of $H$. Let $m$ be the index of $H$. Then $m\leq (c/\delta_*)^{n^2}\leq O_{\alpha,r,\epsilon}(1)$.  Moreover, 
 \[
 \mu(H\backslash AA\inv)\leq \mu(B\backslash AA\inv)<\epsilon^*(\delta,n)\leq \epsilon(m)\mu(H).\qedhere
 \]
 \end{proof}
 
In \cite{PalPFS}, Palac\'{i}n proves a version of the previous result in which $G$ is finite, $A$ is also ``product-free", and $\epsilon$ is constant.\footnote{The actual statement in \cite{PalPFS} asserts  $|H\backslash AA\inv|<\epsilon|G|$, which is trivial since the index of $H$ depends on $\epsilon$; however the proof strategy easily allows for $\epsilon$ to be a function of the index of $H$.}

\begin{remark}\label{rem:Berg}
Recall that a group $G$ is amenable if and only if it admits a (left) F{\o}lner net. In this case, such nets can be used to define the upper and lower \emph{Banach density} of subsets of $G$. In much of the literature from combinatorial number theory, results on amenable groups like those above are often phrased in terms of Banach density. So we note that a translation to this setting can be obtained using the general result that if $G$ is amenable and $A\seq G$, then the upper  (resp., lower) Banach density of $A$ is the supremum (resp., infimum) of $\mu(A)$ over all left-invariant measures $\mu$. Moreover, for a given $A$ the supremum and infimum are both attained (but in general using different measures).  See \cite[Section 2]{HinSt} for details. Via this translation, one can see the relationship between the results of this section and body of work from combinatorial number theory concerning the structure of a product set $AB$, where $A,B$ are subsets of an amenable group (often countable) with positive upper Banach density (see, e.g., \cite{BBF,BFW}). These results are then orthogonal to ours since they focus on products of two distinct sets, but are not uniform in the parameters. (Note that an ultraproduct of amenable groups need not be amenable; see also \cite{OHPo}.)
\end{remark}

\subsection{Remarks on the topological case}
It could be interesting to extend the results above to the topological setting. We give below a step in this direction.  Up to this point the groups $G$ considered in this section were discrete; we now consider locally compact groups. Define a \emph{non-commutative Bohr neighborhood} in a locally compact group $G$ to be the pullback of an open identity neighborhood in a compact group $C$ under a continuous homomorphism from $G$ to $C$. In particular this definition--while it reads the same as the discrete case (Definition \ref{def:Bohr})--now assumes  the homomorphism mentioned there is continuous. Note also that for the purposes of the result below, it is easy to see using Peter-Weyl that $C$ above can be equivalently taken to be a unitary group $\U(n)$ for some $n$.

We assume here a set theory background including some large cardinals, so that all projective sets are universally measurable. See, for example,  \cite[Chapter V]{Kechris}. With a little attention one can no doubt make do with measurability of analytic sets, which is provable in $\mathsf{ZF}$. But to avoid technical difficulties we assume projective determinacy. 

Recall that a locally compact topological group $G$ is  \emph{amenable} if and only if it admits a left-invariant finitely additive probability measure on its Haar-measurable subsets (see \cite{Paterson} for general background on amenablility).

\begin{proposition}\label{prop:topamen}
Assume projective determinacy. Let $G$ be a separable locally compact amenable group with Haar measure algebra $\mathcal{H}$. Let $\mu\colon\mathcal{H}\to[0,1]$ be a left-invariant finitely additive probability measure. Suppose $X\in\mathcal{H}$ is such that $X=X\inv$ and  $\mu(X)>0$. Then $X^4$ contains a non-commutative Bohr neighborhood.
\end{proposition}
\begin{proof}
Since all projective sets are universally measurable, and in particular Haar measurable, any definable set in $(G,\cdot,X)$ is Haar measurable, hence in the domain of $\mu$. Let $(G^*,\cdot,X^*)$ be an $\aleph_1$-saturated elementary extension. We obtain a $\bigwedge$-definable subgroup $\Gamma$ contained in $(X^*)^4$, using Massicot-Wagner \cite{MassWa}. (We note in passing that adding expectation quantifiers, i.e., making the measure definable, remains within the projective set hierarchy, so that $\mu$ may be taken definable as a measure; but this is not needed in \cite{MassWa}.) Now $\Gamma$ is the kernel of a definable homomorphism $\tau^*\colon G^*\to C$ where $C$ is a compact Hausdorff group. Thus the pre-image of an open set is $\bigvee$-definable. By saturation, we have an open identity neighborhood $U$ in $C$ such that $(\tau^*)\inv(U)\seq (X^*)^4$. Let $\tau$ be the restriction of $\tau^*$ to $G$. Then $\tau$ pulls back an open set to a set in the $\Sigma$-algebra of definable sets, so that it is a measurable map. It follows that the homomorphism $\tau$ is actually continuous; see \cite{Kleppner} where this statement is attributed to Banach, and a more general statement is proved. Moreover, $\tau\inv(U)\seq X^4$. Indeed, if $x\in G$ with $\tau(x)\in U$, then $x\in (X^*)^4$; but $G\preceq G^*$, so $x\in X^4$. Thus $\tau\inv(U)$ is a non-commutative Bohr neighborhood contained in $X^4$.
\end{proof}

\begin{remark}
In place of amenability of $G$, we could assume the existence of a left-invariant finitely additive measure on Haar-measurable sets, which is not necessarily finite but satisfies $\mu(X)>0$ and $\mu(X^3)<\infty$. In this case, the proof recovers the softer ``Lie model" part of the beautiful results of Carolino \cite{Carolino-thesis} and Machado \cite{Machado} (but for Borel sets, possibly of Haar measure zero, not necessarily open or closed).
\end{remark}

\begin{remark}
A uniform strengthening of Proposition \ref{prop:topamen} along the lines of Theorem \ref{thm:amenBogo} should be attainable under suitable development of the tools used above. For example, note that Kazhdan's theorem (Theorem \ref{thm:Kaz}) holds in the topological setting. We also point out that in the discrete case, so when $G$ is a countable amenable group, Proposition \ref{prop:topamen} follows from  the main result of \cite{BBF} (proved first for $\Z$ in \cite{BFW}). In particular,  Theorem 3 of \cite{BBF} implies that $X^2$ contains a ``piecewise Bohr" set, and thus $X^4$ contains a noncommutative Bohr neighborhood by an elementary covering argument similar to (1) and (2) in Remark \ref{rem:covering}.
\end{remark}


\begin{thebibliography}{1}

\bibitem{AGG}
M.~A. Alekseev, L.~Y. Glebski\u\i, and E.~I. Gordon, \emph{On approximations of
  groups, group actions and {H}opf algebras}, Zap. Nauchn. Sem. S.-Peterburg.
  Otdel. Mat. Inst. Steklov. (POMI) \textbf{256} (1999), no.~Teor. Predst. Din.
  Sist. Komb. i Algoritm. Metody. 3, 224--262, 268. \MR{1708567}
  
  \bibitem{BaGoPy}
L\'aszl\'o Babai, Albert~J. Goodman, and L\'aszl\'o Pyber, \emph{Groups without
  faithful transitive permutation representations of small degree}, J. Algebra
  \textbf{195} (1997), no.~1, 1--29. \MR{1468882}
  
  
  
  \bibitem{BBF}
M. Beiglb\"ock, V. Bergelson, and A. Fish, \emph{Sumset phenomenon in countable
  amenable groups}, Adv. Math. \textbf{223} (2010), no.~2, 416--432.
  \MR{2565535}


  
  \bibitem{BBHU}
I. Ben~Yaacov, A. Berenstein, C.~W. Henson, and A. Usvyatsov, \emph{Model
  theory for metric structures}, Model theory with applications to algebra and
  analysis. {V}ol. 2, London Math. Soc. Lecture Note Ser., vol. 350, Cambridge
  Univ. Press, Cambridge, 2008, pp.~315--427. \MR{2436146 (2009j:03061)}
  
  
  \bibitem{BFW}
V. Bergelson, H. Furstenberg, and B. Weiss, \emph{Piecewise-{B}ohr sets of
  integers and combinatorial number theory}, Topics in discrete mathematics,
  Algorithms Combin., vol.~26, Springer, Berlin, 2006, pp.~13--37. \MR{2249261}

\bibitem{BjGr}
M. Bj\"{o}rklund and J.~T. Griesmer, \emph{Bohr sets in triple products of
  large sets in amenable groups}, J. Fourier Anal. Appl. \textbf{25} (2019),
  no.~3, 923--936. \MR{3953491}

\bibitem{BGT}
E. Breuillard, B. Green, and T. Tao, \emph{The structure of approximate
  groups}, Publ. Math. Inst. Hautes \'Etudes Sci. \textbf{116} (2012),
  115--221. \MR{3090256}
  


\bibitem{Carolino-thesis}
P.~K. Carolino, \emph{The structure of locally compact approximate groups},
  Ph.D. thesis, University of California--Los Angeles, 2015.


\bibitem{ChPi}
N. Chavarria and A. Pillay, \emph{On pp-elimination and stability in a
  continuous setting}, Ann. Pure Appl. Logic \textbf{174} (2023), no.~5, Paper
  No. 103258, 14. \MR{4554677}
  
 \bibitem{CoBogo}
G. Conant, \emph{On finite sets of small tripling or small alternation in
  arbitrary groups}, Combin. Probab. Comput. \textbf{29} (2020), no.~6,
  807--829. \MR{4173133}
  
    \bibitem{CP-AVSAR}
G. Conant and A. Pillay, \emph{An analytic version of stable arithmetic regularity}, arXiv.2401.14363, 2024.
  
 \bibitem{CPT}
G. Conant, A. Pillay, and C. Terry, \emph{Structure and regularity for subsets
  of groups with finite {VC}-dimension}, J. Eur. Math. Soc. (JEMS) \textbf{24}
  (2022), no.~2, 583--621. \MR{4382479}
  
  
  \bibitem{CurRein}
C.~W. Curtis and I. Reiner, \emph{Representation theory of finite groups and
  associative algebras}, Pure and Applied Mathematics, Vol. XI, Interscience
  Publishers (a division of John Wiley \& Sons, Inc.), New York-London, 1962.
  \MR{0144979}
  
  \bibitem{Dixmier}
J. Dixmier, \emph{Les moyennes invariantes dans les semi-groupes et leurs
  applications}, Acta Sci. Math. (Szeged) \textbf{12} (1950), 213--227.
  \MR{37470}


\bibitem{EllisRL}
R.~L. Ellis, \emph{Extending continuous functions on zero-dimensional spaces},
  Math. Ann. \textbf{186} (1970), 114--122. \MR{261565}


  
 \bibitem{GoLo}
I. Goldbring and V.~C. Lopes, \emph{Pseudofinite and pseudocompact metric
  structures}, Notre Dame J. Form. Log. \textbf{56} (2015), no.~3, 493--510.
  \MR{3373616}
  
    \bibitem{GowQRG}
W.~T. Gowers, \emph{Quasirandom groups}, Combin. Probab. Comput. \textbf{17}
  (2008), no.~3, 363--387. \MR{2410393}

\bibitem{GreenSLAG}
B. Green, \emph{A {S}zemer\'edi-type regularity lemma in abelian groups, with
  applications}, Geom. Funct. Anal. \textbf{15} (2005), no.~2, 340--376.
  \MR{2153903}
  
  \bibitem{Hall}
P. Hall, \emph{Some constructions for locally finite groups}, J. London Math.
  Soc. \textbf{34} (1959), 305--319. \MR{0162845}

\bibitem{Hatcher}
A. Hatcher, \emph{Algebraic topology}, Cambridge University Press, Cambridge,
  2002. \MR{1867354}


\bibitem{HinSt}
N. Hindman and D. Strauss, \emph{Density and invariant means in left amenable
  semigroups}, Topology Appl. \textbf{156} (2009), no.~16, 2614--2628.
  \MR{2561213}


\bibitem{HruAG}
E. Hrushovski, \emph{Stable group theory and approximate subgroups}, J. Amer.
  Math. Soc. \textbf{25} (2012), no.~1, 189--243. \MR{2833482}
  
  \bibitem{Iltis}
R. Iltis, \emph{Some algebraic structure in the dual of a compact group},
  Canad. J. Math. \textbf{20} (1968), 1499--1510. \MR{0232892}




\bibitem{Jech}
Thomas Jech, \emph{Set theory}, second ed., Perspectives in Mathematical Logic,
  Springer-Verlag, Berlin, 1997. \MR{1492987}
  
  \bibitem{Jordan}
C. Jordan, \emph{M{\'e}moire sur les {\'e}quations diff{\'e}rentielles
  line{\'e}aires {\`a} int{\'e}grale alg{\'e}brique}, J. Reine Agnew. Math.
  \textbf{84} (1878), 89--215.
  
 \bibitem{JuschMon}
K. Juschenko and N. Monod, \emph{Cantor systems, piecewise translations and
  simple amenable groups}, Ann. of Math. (2) \textbf{178} (2013), no.~2,
  775--787. \MR{3071509}


  
 \bibitem{Kazh}
D. Kazhdan, \emph{On {$\varepsilon $}-representations}, Israel J. Math.
  \textbf{43} (1982), no.~4, 315--323. \MR{693352}
  
  
  \bibitem{Kechris}
A.~S. Kechris, \emph{Classical descriptive set theory}, Graduate Texts in
  Mathematics, vol. 156, Springer-Verlag, New York, 1995. \MR{1321597}

\bibitem{Kleppner}
A. Kleppner, \emph{Measurable homomorphisms of locally compact groups}, Proc.
  Amer. Math. Soc. \textbf{106} (1989), no.~2, 391--395. \MR{948154}


\bibitem{Knappbook}
A.~W. Knapp, \emph{Lie groups beyond an introduction}, second ed., Progress in
  Mathematics, vol. 140, Birkh\"{a}user Boston, Inc., Boston, MA, 2002.
  \MR{1920389}
  
  \bibitem{Kuratowski}
K. Kuratowski, \emph{Topology. {V}ol. {I}.}, new ed., Academic Press, New
  York-London; Pa\'{n}stwowe Wydawnictwo Naukowe [Polish Scientific
  Publishers], Warsaw,, 1966, Translated from the French by J. Jaworowski.
  \MR{217751}
  



\bibitem{LOST}
M.~W. Liebeck, E.~A. O'Brien, A. Shalev, and P.~H. Tiep, \emph{The {O}re
  conjecture}, J. Eur. Math. Soc. (JEMS) \textbf{12} (2010), no.~4, 939--1008.
  \MR{2654085}

\bibitem{Machado}
S. Machado, \emph{Closed approximate subgroups: compactness, amenability and
  approximate lattices}, arXiv:2011.01829, 2020.



  
  \bibitem{MarZelFSG}
C.~Martinez and E.~Zelmanov, \emph{Products of powers in finite simple groups},
  Israel J. Math. \textbf{96} (1996), no.~part B, 469--479. \MR{1433702}
  
 \bibitem{MassWa}
J.-C. Massicot and F.~O. Wagner, \emph{Approximate subgroups}, J. \'Ec.
  polytech. Math. \textbf{2} (2015), 55--64. \MR{3345797}



\bibitem{MOS}
S. Montenegro, A. Onshuus, and P. Simon, \emph{Stabilizers, {${\rm NTP}_2$}
  groups with f-generics, and {PRC} fields}, J. Inst. Math. Jussieu \textbf{19}
  (2020), no.~3, 821--853. \MR{4094708}
  
 \bibitem{Newman}
M. Newman, \emph{Two classical theorems on commuting matrices}, J. Res. Nat.
  Bur. Standards Sect. B \textbf{71B} (1967), 69--71. \MR{220748}

\bibitem{NiPy}
N. Nikolov and L. Pyber, \emph{Product decompositions of quasirandom groups and
  a {J}ordan type theorem}, J. Eur. Math. Soc. (JEMS) \textbf{13} (2011),
  no.~4, 1063--1077. \MR{2800484}

\bibitem{NST}
N. Nikolov, J. Schneider, and A. Thom, \emph{Some remarks on finitarily
  approximable groups}, J. \'Ec. polytech. Math. \textbf{5} (2018), 239--258.
  \MR{3749196}
  
 \bibitem{NiSe}
N. Nikolov and D. Segal, \emph{Generators and commutators in finite groups;
  abstract quotients of compact groups}, Invent. Math. \textbf{190} (2012),
  no.~3, 513--602. \MR{2995181}
  
 \bibitem{OHPo}
A. Ould~Houcine and F. Point, \emph{Alternatives for pseudofinite groups}, J.
  Group Theory \textbf{16} (2013), no.~4, 461--495. \MR{3100996}

\bibitem{PalPFS}
D. Palac\'{\i}n, \emph{On compactifications and product-free sets}, J. Lond.
  Math. Soc. (2) \textbf{101} (2020), no.~1, 156--174. \MR{4072489}

\bibitem{Paterson}
A.~L.~T. Paterson, \emph{Amenability}, Mathematical Surveys and Monographs,
  vol.~29, American Mathematical Society, Providence, RI, 1988. \MR{961261}
  
 \bibitem{Pibook}
A. Pillay, \emph{An introduction to stability theory}, Oxford Logic Guides,
  vol.~8, The Clarendon Press Oxford University Press, New York, 1983.
  \MR{719195 (85i:03104)}
  
  \bibitem{PilCLG}
A. Pillay, \emph{Type-definability, compact {L}ie groups, and o-minimality}, J.
  Math. Log. \textbf{4} (2004), no.~2, 147--162. \MR{2114965}

\bibitem{PiRCP}
A. Pillay, \emph{Remarks on compactifications of pseudofinite groups}, Fund.
  Math. \textbf{236} (2017), no.~2, 193--200. \MR{3591278}
 
\bibitem{Ruz94}
I.~Z. Ruzsa, \emph{Generalized arithmetical progressions and sumsets}, Acta
  Math. Hungar. \textbf{65} (1994), no.~4, 379--388. \MR{1281447}
 

 
  
  \bibitem{SaxWilFSG}
Jan Saxl and John~S. Wilson, \emph{A note on powers in simple groups}, Math.
  Proc. Cambridge Philos. Soc. \textbf{122} (1997), no.~1, 91--94. \MR{1443588}
  
\bibitem{SchTh-lattice}
J. Schneider and A. Thom, \emph{A note on the normal subgroup lattice of
  ultraproducts of finite quasisimple groups}, Proc. Amer. Math. Soc.
  \textbf{149} (2021), no.~5, 1929--1942. \MR{4232187}

\bibitem{SchurJST}
I. Schur, \emph{Uber {G}ruppen periodischer linearer {S}ubstitutionen},
  Sitzber. Preuss. Akad. Wiss. (1911), 619--627.

\bibitem{StoTh}
Abel Stolz and Andreas Thom, \emph{On the lattice of normal subgroups in
  ultraproducts of compact simple groups}, Proc. Lond. Math. Soc. (3)
  \textbf{108} (2014), no.~1, 73--102. \MR{3162821}

\bibitem{Szarek}
S.~a.~J. Szarek, \emph{Metric entropy of homogeneous spaces}, Quantum
  probability ({G}da\'{n}sk, 1997), Banach Center Publ., vol.~43, Polish Acad.
  Sci. Inst. Math., Warsaw, 1998, pp.~395--410. \MR{1649741}

\bibitem{TaoICM}
T. Tao, \emph{The dichotomy between structure and randomness}, presented at the
  2006 International Congress of Mathematicians, Madrid, available at:
  \url{https://www.math.ucla.edu/~tao/preprints/Slides/icmslides2.pdf}.

\bibitem{TaoH5P}
T. Tao, \emph{Hilbert's fifth problem and related topics}, Graduate Studies in
  Mathematics, vol. 153, American Mathematical Society, Providence, RI, 2014.
  \MR{3237440}

\bibitem{TaoVu}
T. Tao and V. Vu, \emph{Additive combinatorics}, Cambridge Studies in Advanced
  Mathematics, vol. 105, Cambridge University Press, Cambridge, 2006.
  \MR{2289012}
  
 \bibitem{TeZi}
K. Tent and M. Ziegler, \emph{A course in model theory}, Lecture Notes in
  Logic, vol.~40, Association for Symbolic Logic, La Jolla, CA, 2012.
  \MR{2908005}
  
 \bibitem{Tits}
J. Tits, \emph{Free subgroups in linear groups}, J. Algebra \textbf{20} (1972),
  250--270. \MR{286898}

\bibitem{Turing}
A.~M. Turing, \emph{Finite approximations to {L}ie groups}, Ann. of Math. (2)
  \textbf{39} (1938), no.~1, 105--111. \MR{1503391}

\bibitem{WilSPfG}
John~S. Wilson, \emph{On simple pseudofinite groups}, J. London Math. Soc. (2)
  \textbf{51} (1995), no.~3, 471--490. \MR{1332885}
  
\end{thebibliography}
\end{document}